%% file: Erroranalysis-nonlinear-QSE-StrangsplittingPM.tex
\documentclass[sn-mathphys,Numbered]{sn-jnl}

\usepackage{graphicx}%
\usepackage{multirow}%
\usepackage{amsmath}
\usepackage{amssymb,amsfonts,subfigure}%
\usepackage{mathrsfs,bm,float,geometry,array,booktabs}
\usepackage{multirow}
\usepackage{amsthm}%
\usepackage{mathrsfs}%
\usepackage[title]{appendix}%
\usepackage{xcolor}%
\usepackage{textcomp}%
\usepackage{manyfoot}%
\usepackage{booktabs}%
\usepackage{algorithm}%
\usepackage{algorithmicx}%
\usepackage{algpseudocode}%
\usepackage{listings}%

\newtheorem{theorem}{Theorem}
\newtheorem{proposition}[theorem]{Proposition}%
\newtheorem{remark}{Remark}%
\usepackage[justification=centering]{caption}

\input{ex_shared_NM}

\begin{document}
	
	\title[Article Title]{Convergence analysis of time-splitting projection method for nonlinear quasiperiodic Schr\"{o}dinger equation}
	
	
	\author[1]{\fnm{Kai} \sur{Jiang}}\email{kaijiang@xtu.edu.cn}
	
	\author[2]{\fnm{Shifeng} \sur{Li}}\email{shifengli@whu.edu.cn}
	
	\author[3]{\fnm{Xiangcheng} \sur{Zheng}}\email{xzheng@sdu.edu.cn}
	
	\affil[1]{Department of Mathematics and Computational
		Science, Xiangtan University, Xiangtan, Hunan,
		411105, P. R. China}
	\affil[2]{School of Mathematics and Statistics, Wuhan University, Wuhan, 430072, China}
	\affil[3]{School of Mathematics, Shandong University, Jinan, Shandong 250100, China}
	
	%
	
	
	\abstract{This work proposes and analyzes an efficient numerical method for solving the nonlinear Schr\"odinger equation with quasiperiodic potential, where the projection method is applied in space to account for the quasiperiodic structure and the Strang splitting method is used in time. 
		While the transfer between spaces of low-dimensional quasiperiodic and high-dimensional periodic functions and its coupling with the nonlinearity of the operator splitting scheme make the analysis more challenging. 
		Meanwhile, compared to conventional numerical analysis of periodic Schr\"odinger systems, many of the tools and theories are not applicable to the quasiperiodic case. We address these issues to prove the spectral accuracy in space and the second-order accuracy in time. Numerical experiments are performed to substantiate the theoretical findings.
	}

	
	
	
	\maketitle
	
	\section{Introduction}
	\label{sec:Intro}
	The nonlinear Schr\"{o}dinger equation with the cubic nonlinearity is a universal mathematical model describing many physical phenomena, and there exist extensive investigations for this model with periodic potentials or given boundary conditions, see e.g. \cite{Schrodinger1926undulatory,
		Thalhammer2012Convergence,Auzinger2015Defect,Dirac1958principles,Gauckler2010Nonlinear,
		Gauckler2010Convergence,Jahnke2000error,
		LiWu,LiZha}.
	In recent years, the quasiperiodic potential $V$ is increasingly considered in various fields. 
	For instance, the Schr\"{o}dinger system with a potential function $V=V_1+V_2$ is typically considered in Moir\'e lattices \cite{Ablowitz2006Solitons,Bagci2021Soliton,Wang2020Localization}, where $V_1$ represents a periodic lattice and $V_2$ is a rotated $V_1$ by a twisted angle. Varying this twisted angle could result in a transformation between periodic and quasiperiodic structures, as well as a state transition between localization and delocalization in these systems.
	In fact, the potential function $V$ often appears incommensurate, which is crucial to the construction of quasiperiodic potentials. Even more surprisingly, the nonlinear quasiperiodic Schr\"odinger systems also exhibit more intriguing phenomena of significant scientific value, including metal-insulator transitions, defects, Anderson localization, dislocations and aperiodic lattice solitons \cite{Aubry1980Effects,Fu2020Optical,li2021numerical,Shi2024Anderson}.
	

	Motivated by the above discussions, we consider the $d$-dimensional nonlinear quasiperiodic Schr\"odinger equation (NQSE) with the real-valued smooth quasiperiodic potential function $V(\bx)$
	\begin{align}
		\begin{cases}
			i\dfrac{\partial \psi(\bx,t)}{\partial t}=-\Delta \psi(\bx,t)+V(\bx)\psi(\bx,t)+\theta \vert \psi(\bx,t)\vert^2 \psi(\bx,t) ,~~(\bx,t)\in\bbR^d\times[0,T],\\
			\psi(0)=\psi(\bx, 0),
		\end{cases}
		\label{eq:QSE}
	\end{align}
	where $\Delta= \sum_{j=1}^d \partial_{x_j x_j}$ is the Laplace operator and the cubic nonlinear term originates from the nonlinear (Kerr) change of the
	refractive index with $\vert \psi(\bx,t)\vert^2= \psi(\bx,t)\bar{\psi}(\bx,t)$. The parameter $\theta\in \bbR$  describes the strength of the nonlinearity.
	
	
	
	
	There has been some improvement in the research on the mathematical theory of time- or space-quasiperiodic Schr\"odinger systems. For instance, Avila \textit{et al.} made significant contributions to the spectral theory of quasiperiodic Schr\"odinger operators (QSOs) from the perspective of dynamical systems \cite{Avila2009spectrum,Avila2015Sharp,Berti2012Sobolev,Bourgain1994Construction,Marx2017Dynamics}. Wang \textit{et al.} investigated Anderson localization of QSOs and the existence of solutions to quasiperiodic Schr\"odinger equations (QSEs) \cite{Shi2021Absence,Wang2020Space,Wang2022Infinite}. 
	Although research on various aspects of quasiperiodic Schr\"odinger systems has attracted significant attention in recent years, many issues remain unresolved in the study of these systems, including the spectral theory of high-dimensional QSOs, the well-posedness of quasiperiodic solutions for various Schr\"odinger equations, and the development of associated numerical algorithms. Numerically, quasiperiodic systems lack translation symmetry, have no boundary and are not decay, which introduce significant challenges. Therefore, our goal is not only to provide an effective numerical algorithm but also to establish rigorous theoretical analysis that ensures the reliability of our method, thereby advancing the development of quasiperiodic Schr\"odinger systems in mathematics and related application fields. There are several numerical methods to solve quasiperiodic systems, such as periodic approximation method (PAM), quasiperiodic spectral method (QSM) and the projection method (PM).
	PAM is a commonly used approach for solving quasiperiodic systems \cite{Ablowitz2006Solitons,Wang2020Localization,Fu2020Optical}. Its core idea is to approximate a quasiperiodic system by a periodic function system within a finite domain. While relatively mature research mechanisms exist for these approximation systems, this algorithm still faces unavoidable rational approximation errors and is constrained by a limited approximation domain \cite{JiaZha,Jiang2022approximation}.
	To resolve this issue, QSM and PM have been developed \cite{JiaZha} and applied to solve the linear quasiperiodic Schr\"odinger equation \cite{jiang2023High-accuracy}, which corresponds to model (\ref{eq:QSE}) with $\theta=0$. For the case $\theta\neq 0$, efficient numerical method and rigorous numerical analysis remain untreated.
	
	This work proposes and analyzes an efficient numerical method for solving the nonlinear Schr\"odinger equation with quasiperiodic potential, where PM is applied in space to account for the quasiperiodic structure, and the Strang splitting method is used in time, see e.g. \cite{Bao2003Numerical,Bao2021Uniform,Wu2024Error,Einkemmer2013almost,Einkemmer2014Convergence}. 
	Compared to the numerical analysis of periodic Schr\"odinger equations,
	the analysis of algorithms for solving NQSEs presents greater challenges. One major challenge arises from the relatively dense Fourier frequencies of quasiperiodic functions, for which some conventional analytical tools and theories, such as the Sobolev embedding theorems, are no longer valid in the numerical analysis of NQSEs. To address this, we utilize the regularity of the corresponding parent functions to control the regularity of quasiperiodic functions. 
	Meanwhile, novel strategies have been introduced in the analytical process. For example, we devise an auxiliary high-dimensional periodic function system (See \eqref{eq:QSE-parent} for details) to manage the upper bound of the numerical quasiperiodic solutions at each iteration step.

	The rest of the paper is organized as follows: In Section \ref{sec2}, several preliminary results are presented or proved.
	Section \ref{sec:NA} introduces the numerical scheme and its error analysis results. Section \ref{sec:analysis} provides miscellaneous auxiliary estimates that support the proof of the main theorem. Section \ref{sec:Num} presents numerical results that validate the theoretical analysis. Finally, a conclusion is presented in Section \ref{sec:diss}.

	\section{Preliminaries}\label{sec2}
	
	\subsection{Notations}
	Let $\bbQ$ be the set of rational numbers. We recall the definition of the quasiperiodic function.
	\begin{defy}
		A matrix $\bP\in\bbR^{d\times n}$ is the projection matrix, if it belongs to the set $\mathbb P
		:=\{\bP=(\bp_1,\cdots,\bp_n)\in\bbR^{d\times n}: \bp_1,\cdots,\bp_n  ~\mbox{are~} \bbQ\mbox{-linearly independent}\}.$
	\end{defy}
	\begin{defy}
		\label{def:quasiperiodic}
		A $d$-dimensional function $\psi(\bx)$ is quasiperiodic if there exists an $n$-dimensional periodic function $\psi_p$ and a projection matrix $\bP\in \mathbb P^{d\times n}$, such that $\psi(\bx)=\psi_p(\bP^T \bx)$ for all $\bx\in\bbR^d$.
	\end{defy}
	
	With this definition, let $\QP(\bbR^d)$ be the space of all $d$-dimensional quasiperiodic functions. For convenience, $\psi_p$ in Definition \ref{def:quasiperiodic} is called the parent function of $\psi$. Next, we present some relations between the quasiperiodic functions and their parent functions.

	\begin{proposition}
		\label{pro:projectmatrix}
		For quasiperiodic functions $\phi,\varphi\in \QP(\bbR^d)$, $\phi_p$ and $\varphi_p$ are the $n_1$- and $n_2$-dimensional parent functions of $\phi$ and $\varphi$, respectively. Assume that there exist projection matrices 
		\begin{align*}
			\bP_1=(\alpha_1,\cdots,\alpha_{n_1})\in\bbR^{d\times n_1},~~
			\bP_2=(\beta_1,\cdots,\beta_{n_2})\in\bbR^{d\times n_2},
		\end{align*}
		such that $\phi(\bx)=\phi_p(\bP_1^T\bx)$ and $\varphi(\bx)=\varphi_p(\bP_2^T\bx)$. Then for the projection matrix $\bP=(\gamma_1,\cdots,\gamma_{n_3}),$ where $\{\gamma_1,\cdots,\gamma_{n_3}\}
		\subset\{\alpha_1,\cdots,\alpha_{n_1},\beta_1,\cdots,\beta_{n_2}\}$ is the largest  $\bbQ$-linearly independent vector set and $n_3 \geq \max\{n_1,n_2\}$, there exist $n_3$-dimensional periodic functions $\check \phi_p$ and $\check \varphi_p$ such that
		\begin{align*}
			\phi(\bx)=\check \phi_p(\bP^T\bx),~~\varphi(\bx)=\check \varphi_p(\bP^T\bx).
		\end{align*}
		Moreover, it follows that
		
		(i) For $\psi=\phi+\varphi$, the parent function $\psi_p$ of $\psi$ is $\psi_p=\check \phi_p+\check \varphi_p$.
		
		(ii) For $\psi=\phi\varphi$, the parent function $\psi_p$ of $\psi$ is $\psi_p=\check \phi_p\check \varphi_p$.
	\end{proposition}
	
	\begin{proof}
		Without loss of generality, we set
		$
		\gamma_1=\alpha_1,\cdots,\gamma_{n_1}=\alpha_{n_1},\gamma_{n_1+1}=\beta_1,\cdots,\gamma_{n_3}=\beta_s,
		$
		with $0\leq s\leq n_2$. Then for $s+1\leq j \leq n_2$, we have
		$
		\beta_j=\sum_{k=1}^{n_3}a_{j,k}\gamma_{k}$ for some $a_{j,k}\in \bbQ$.
		Then
		\begin{align*}
			\phi(\bx)=\phi_p(\alpha_1^T\bx,\cdots,\alpha_{n_1}^T\bx)=\phi_p(\gamma_1^T\bx,\cdots,\gamma_{n_1}^T\bx)
			=\check \phi_p(\gamma_1^T\bx,\cdots,\gamma_{n_3}^T\bx),
		\end{align*}
		where $\check \phi_p(y_1,\cdots,y_{n_3})=\phi_p(y_1,\cdots,y_{n_1})$.
		Similarly,
		\begin{align*}
			\varphi(\bx)
			&=\varphi_p(\beta_1^T\bx,\cdots,\beta_{n_2}^T\bx)\\
			&=\varphi_p(\gamma_{n_1+1}^T\bx,\cdots,\gamma_{n_1+s}^T\bx,\sum_{k=1}^{n_3}a_{{n_1+s+1},k}\gamma_{k}^T\bx,\cdots,
			\sum_{k=1}^{n_3}a_{{n_2},k}\gamma_{k}^T\bx)\\
			&=\check \varphi_p(\gamma_1^T\bx,\cdots,\gamma_{n_3}^T\bx),
		\end{align*}
		where $\check \varphi_p(y_1,\cdots,y_{n_3})
		=\varphi_p(y_{n_1+1},\cdots,y_{n_1+s},\sum_{k=1}^{n_3}a_{{n_1+s+1},k}y_k,\cdots,\sum_{k=1}^{n_3}a_{{n_2},k}y_{k})$. Thus the first statement of this proposition is proved, and the statements (i)--(ii) are immediately obtained.
	\end{proof}

	\begin{remark}
		In the following analysis, we do not distinguish the parent functions $\phi_p$ and $\check \phi_p$ of a quasiperiodic function $\phi$.
		Furthermore, for the quasiperiodic function $\psi=\phi \varphi$, denote $\psi_p=(\phi \varphi)_p$ as its parent function.
	\end{remark}

	Let $K_T=\{\bx=(x_j)_{j=1}^d\in\bbR^d: \vert x_j\vert \leq T, ~j=1,\cdots,d\}$ be the cube in $\bbR^d$. 
	The mean value $\calM\{\psi(\bx)\}$ of $\psi\in \QP(\bbR^d)$ is defined as
	\begin{align*}
		\calM\{\psi(\bx)\}=\lim_{T\rightarrow +\infty}
		\frac{1}{(2T)^d}\int_{\bs+K_T} \psi(\bx)\,d\bx
		:=\bbint  \psi(\bx)\,d\bx,
	\end{align*}
	where the limit exists uniformly for all $\bs\in\bbR^d$ \cite{Corduneanu1989Almost}. An elementary calculation shows
	\begin{align*}
		\calM\{e^{i\blam\cdot \bx}e^{-i\bm\beta \cdot \bx}\}
		=\begin{cases}
			1, ~~\blam=\bm\beta,\\
			0, ~~ \blam\neq \bm\beta.
		\end{cases}
	\end{align*}
	Correspondingly, the continuous Fourier-Bohr transform of $\psi(\bx)$ is
	\begin{align}
		\hpsi_{\blam}= \calM\{\psi(\bx)e^{-i\blam\cdot  \bx}\},
		\label{eq:transform-FC}
	\end{align}
	where $\blam\in\bbR^d$.
	The Fourier series for $\psi(\bx)$ is
	\begin{align*}
		\psi(\bx)\sim\sum_{j=1}^{\infty} \hpsi_{\blam_j} e^{i\blam_j\cdot \bx},
	\end{align*}
	where $\blam_j\in \sigma(\psi)=\{\blam: \blam = \bP\bk,~\bk\in \bbZ^n \}$ are Fourier exponents and $\hpsi_{\blam_j}$ defined in \eqref{eq:transform-FC} are Fourier coefficients. 
	Furthermore, the Parseval identity of the quasiperiodic function is
	\begin{align}
		\calM \{\vert \psi\vert^2 \}
		=\sum_{\blam\in \sigma(\psi)} \vert \hpsi_{\blam}\vert^2.
		\label{eq:Parseval}
	\end{align}

	\subsection{Function spaces}
	We first introduce the quasiperiodic function spaces on $\bbR^d$.
	\begin{itemize}
		\item {\bf $L^q_{QP}(\bbR^d)$ space}: For any fixed $q\in [1,\infty)$, denote 
		\begin{align*}
			L^q_{QP}(\bbR^d)=\Big \{\psi(\bx)\in\QP(\bbR^d):~\Vert \psi \Vert^q_{q}
			=\bbint \vert \psi(\bx) \vert^q\,d\bx  < \infty \Big \},
		\end{align*}
		and
		\begin{align*}
			L^{\infty}_{QP}(\bbR^d)=\{\psi(\bx)\in\QP(\bbR^d):~\Vert \psi \Vert_{\infty}=\sup_{\bx\in \bbR^d}\vert \psi(\bx) \vert < \infty \}.
		\end{align*}
		In particular, the $\|\cdot\|_2$ norm could be expressed via the inner product $(\cdot, \cdot)_{L_{QP}^2(\bbR^d)}$ with
		\begin{align*}
			(\psi, \varphi)_{L^2_{QP}(\bbR^d)}=\bbint  \psi(\bx)\bar{\varphi}(\bx)\,d\bx.
		\end{align*}
		By applying the Parseval identity \eqref{eq:Parseval}, we have
		\begin{align*}
			\Vert \psi \Vert_{L^2_{QP}(\bbR^d)}^2=\sum_{\blam\in\sigma(\psi)}\vert \hpsi_{\blam} \vert^2.
		\end{align*}
		
		
		\item {\bf $\calC^{\alpha}_{QP}(\bbR^d)$ space:} For $\alpha\in\bbN$, the space $\calC^{\alpha}_{QP}(\bbR^d)$ consists of quasiperiodic functions with continuous derivatives up to order $\alpha$ on $\bbR^d$.
		The $\calC^{\alpha}_{QP}$-norm of $\psi\in \calC^{\alpha}_{QP}(\bbR^d)$ is defined by
		\begin{align*}
			\Vert \psi \Vert_{\calC^{\alpha}_{QP}} =\sum_{\vert \bmm\vert \leq \alpha} \sup_{\bx \in \bbR^d}  \vert \partial_{\bx}^{\bmm} \psi \vert.  
		\end{align*}
		
		\item {\bf $H^{\alpha}_{QP}(\bbR^d)$ space:} For $\alpha\in\bbN$, $H^{\alpha}_{QP}(\bbR^d)$ comprises all quasiperiodic functions with $L^2_{QP}$ partial derivatives up to order $\alpha$. For $\psi,\varphi\in H^{\alpha}_{QP}(\bbR^d)$, the inner product $(\cdot, \cdot)_{H^{\alpha}_{QP}(\bbR^d)}$ is
		\begin{align*}
			(\psi, \varphi)_{H^{\alpha}_{QP}(\bbR^d)}=(\psi, \varphi)_{L^2_{QP}(\bbR^d)}+ \sum_{\vert
				\bmm\vert=\alpha}(\partial^{\bmm}_{\bx} \psi, \partial^{\bmm}_{\bx} \varphi)_{L^2_{QP}(\bbR^d)}.
		\end{align*}
		The corresponding norm is
		\begin{align*}
			\Vert \psi \Vert_{H^{\alpha}_{QP}(\bbR^d)}^2=\sum_{\blam\in\sigma(\psi)}(1+\vert \blam\vert^2)^{\alpha}\vert \hpsi_{\blam} \vert^2,
		\end{align*}
		where $\vert \blam \vert= \sum^d_{j=1} \vert \lambda_j\vert$.
		In particular, for $\alpha=0$, $H^0_{QP}(\bbR^d)=L^2_{QP}(\bbR^d)$.
		To simplify the notations, we denote  $(\cdot, \cdot)=(\cdot, \cdot)_{L^2_{QP}(\bbR^d)}
		$ and $(\cdot, \cdot)_{\alpha}=(\cdot, \cdot)_{H^{\alpha}_{QP}(\bbR^d)}$.
	\end{itemize}
	
	To introduce the periodic function spaces on the $n$-dimensional torus $\bbT^n=\bbR^n/2\pi \bbZ^n$, define the Fourier transform of $U(\by)$  on $\bbT^n$
	\begin{align}
		\hU_{\bk}
		=\frac{1}{\vert \bbT^n\vert}\int_{\bbT^n}e^{-i\bk\cdot \by}U(\by)\,d\by,~~\bk\in\bbZ^n.
		\label{eq:rasiedFC}
	\end{align} 
	
	\begin{itemize}
		\item {\bf $L^2(\bbT^n)$ space:}
		$L^2(\bbT^n)=\Big\{U(\by): \frac{1}{|\bbT^n|}\int_{\bbT^n}\vert U\vert^2\,d\by < +\infty\Big\},$
		equipped with inner product
		\begin{align*}
			(U_1, U_2)_{L^2(\bbT^n)}=\frac{1}{|\bbT^n|}\int_{\bbT^n}U_1\overline{U}_2\,d\by,
		\end{align*}
		and the norm $\|U\|^2:=(U,U)_{L^2(\bbT^n)}$.
		
		
		\item {\bf $X_{\alpha}(\bbT^n)$ space:} For any $U\in L^2(\bbT^n)$ with, the Fourier series expansion 
		\begin{align*}
			U(\by)=\sum_{\bk\in\bbZ^n} \hU_{\bk}e^{i\bk\cdot \by},
		\end{align*}
		the linear operator $(-\Delta)^\alpha$ with $\alpha\in\bbR$ is given by
		\begin{align*}
			(-\Delta)^\alpha U
			=\sum_{\bk\in\bbZ^n} \Vert \bk \Vert^{2\alpha} \hU_{\bk}  e^{i\bk\cdot \by},~~\Vert \bk \Vert^2=\sum_{i=1}^{n} \vert k_i \vert^2,
		\end{align*}
		and the corresponding space $X_{\alpha}(\bbT^n)$ is defined as
		\begin{align*}
			X_{\alpha}(\bbT^n)=
			\Big \{U(\by)=\sum_{\bk\in\bbZ^n} \hU_{\bk}e^{i\bk\cdot \by}\in L^2(\bbT^n):
			\Vert  (-\Delta)^\alpha U \Vert^2
			=\sum_{\bk\in\bbZ^n} \vert \hU_{\bk}\vert^2\cdot
			\Vert \bk \Vert^{4\alpha} <\infty \Big \}.
		\end{align*}
		The $X_{\alpha}(\bbT^n)$ forms a Hilbert space with inner product 
		\begin{align*}
			(U,W)_{X_{\alpha}}=(U,W)_{L^2(\bbT^n)}
			+((-\Delta)^\alpha U,(-\Delta)^\alpha W)_{L^2(\bbT^n)},
		\end{align*}
		the norm and semi-norm with $\alpha\neq 0$ are 
		\begin{align*}
			\Vert U\Vert^2_{X_{\alpha}}=\sum_{\bk\in\bbZ^n} 
			(1+\Vert \bk \Vert^{4\alpha}) \vert \hU_{\bk} \vert^2,~~
			\vert U \vert^2_{X_{\alpha}}=\sum_{\bk\in\bbZ^n} 
			\Vert \bk \Vert^{4\alpha} \vert \hU_{\bk} \vert^2,
		\end{align*}
		where $\Vert \bk\Vert^2=\sum^n_{j=1} \vert k_j\vert^2$. When $\alpha=0$, $\Vert U\Vert_{X_0}=\Vert U\Vert$.
		
		
	\end{itemize}

	\subsection{Basic properties of quasiperiodic functions}
	
	The following lemma states that the quasiperiodic Fourier coefficients $\hpsi_{\bk}$ of \eqref{eq:transform-FC} are equal to their parent Fourier coefficients $\hpsi_{p,\bk}$ of \eqref{eq:rasiedFC}.
	Throughout this work, $C$ denotes a generic positive constant that may be different at different occurrences.
	
	\begin{lemma}
		\cite{Jiang2024Numerical}
		\label{lemma:Birkhoff}
		For a given quasiperiodic function
		\begin{align*}
			\psi(\bx)=\psi_p(\bP^T \bx), ~~\bx\in\bbR^d,
		\end{align*}
		where $\psi_p(\by)$ is its parent function defined on the tours $\bbT^n$, $\bP$ is the projection matrix and $\hpsi_{p,\bk}$ is the continuous Fourier coefficient of $\psi_p$, we have
		\begin{align*}
			\hpsi_{\blam}=\hpsi_{p,\bk},~~\mbox{with}~~\blam =\bP \bk,~~\bk\in\bbZ^n.
		\end{align*}
	\end{lemma}

	Based on the isomorphic relationship between the quasiperiodic function and its parent function, we give the following norm inequalities.
	\begin{lemma}
		\cite{jiang2023High-accuracy}
		\label{lem:normineq}
		For any $\psi\in \QP(\bbR^d)$ and its $n$-dimensional parent function $\psi_p$, the following estimates hold for $\alpha> n/4$ and $\psi_p\in X_{\alpha}$
		\begin{align}
			\Vert \psi\Vert_{L_{QP}^{\infty}(\bbR^d)} \leq C\, \Vert \psi_p\Vert_{X_{\alpha}}
			\label{eq:inftyembedX};~~\Vert \varphi \psi \Vert \leq C \, \Vert \varphi \Vert \cdot\Vert \psi_p\Vert_{X_\alpha},
			~~\varphi \in L_{QP}^{2}(\bbR^d).
		\end{align}
	\end{lemma}

	

	Next, the maximum nonzero singular value of $\bP$ is used to control the regularity relation between the quasiperiodic function and the corresponding parent function.
	
	\begin{lemma}
		\label{pro:bound-deltav}
		For $\psi\in L^2_{QP}(\bbR^d)$ with the corresponding parent function $\psi_p$ and the projection matrix $\bP$, we have
		\begin{align*}
			\Vert \Delta \psi\Vert \leq \sigma_{max}^2(\bP) \Vert \Delta \psi_p \Vert,
		\end{align*}
		where $\sigma_{max}(\bP)$ is the maximum nonzero singular value of $\bP$.
	\end{lemma}
	
	\begin{proof}
		For $\psi(\bx)=\sum_{\blam\in \sigma(\psi)} \hpsi_{\blam} e^{i \blam \cdot \bx}$,
		we have
		\begin{align*}
			\Vert \Delta \psi\Vert^2 
			=\Big \Vert \sum_{\blam\in \sigma(\psi)} \Vert \blam \Vert^2_2 \,\hpsi_{\blam} e^{i \blam \cdot \bx} \Big \Vert^2
			&= \sum_{\blam\in \sigma(\psi)}  \Vert \bP\bk \Vert_2^4 \Vert \hpsi_{\blam} \Vert^2\\
			&\leq \sigma_{max}^4 (\bP)\sum_{\blam\in \sigma(\psi)}  \Vert \bk \Vert_2^2 \Vert \hpsi_{\blam} \Vert^2 \\
			&= \sigma_{max}^4 (\bP) \Vert \Delta \psi_p\Vert^2.
		\end{align*}
	\end{proof}
	
	\begin{lemma} \cite{jiang2023High-accuracy}
		\label{thm:selfadjoint}
		The operator $\Delta$ is self-adjoint in the $L^2_{QP}$ inner product.
	\end{lemma}
	
	\section{Numerical scheme and error estimate}
	\label{sec:NA}
	
	We apply the Strang splitting method in time and the PM in space to construct the fully discrete scheme of the NQSE \eqref{eq:QSE}. We then present the main result of the convergence analysis.

	\subsection{Strang splitting method}
	First, we divide NQSE \eqref{eq:QSE} into two subproblems.
	The first subproblem is
	\begin{align*}
		i\frac{\partial \phi(\bx,t)}{\partial t}=-\Delta \phi(\bx,t),~~\phi_0=\phi(\bx, 0),~~0\leq t\leq T,
	\end{align*}
	with the solution
	$\phi(\bx,t)=e^{it\Delta}\phi_0=\calF(t)\phi_0$.
	The second subproblem is 
	\begin{align}
		i\frac{\partial \varphi(\bx,t)}{\partial t}=B(\varphi(\bx,t)) \varphi(\bx,t),~~\varphi_0=\varphi(\bx, 0),~~0\leq t\leq T.
		\label{eq:secondproblem}
	\end{align}
	where $B(\varphi)=V(\bx)+\theta\vert \varphi\vert^2$. Since the potential function $V(\bx)$ is real-valued, we have
	\begin{align*}
		i \varphi_t(\bx,t) \bar{\varphi}(\bx,t)=V(\bx) \vert \varphi(\bx,t)\vert^2+\theta \vert \varphi(\bx,t)\vert^4,
	\end{align*}
	and
	\begin{align*}
		-i\bar{\varphi}_t(\bx,t)\varphi(\bx,t)=V(\bx) \vert \varphi(\bx,t)\vert^2+\theta \vert \varphi(\bx,t)\vert^4.
	\end{align*}
	Thus, $\varphi_t(\bx,t) \bar{\varphi}(\bx,t)=-\bar{\varphi}_t(\bx,t)\varphi(\bx,t),$
	\textit{i.e.,}
	$\partial_t \vert \varphi(\bx,t)\vert^2= 2\Re (\varphi_t \bar{\varphi})=0$, which means 
	$B(\varphi(\bx,t))=B(\varphi_0).$ Therefore, we can obtain the analytical solution of \eqref{eq:secondproblem}
	\begin{align*}
		\varphi(\bx, t)=\calS(t)\varphi_0:=e^{-itB(\varphi_0)}\varphi_0,~~0\leq t\leq T.
	\end{align*}
	
	For an integer $M>0$, set the time size $\tau=T/M$ and $t_m=m\tau$ for $1\leq m\leq M$, then the strang splitting method for the NQSE (\ref{eq:QSE}) at time $t_m$ is
	\begin{align}
		\psi_m=\calF(\tau/2) \circ \calS(\tau) \circ \calF(\tau/2)\psi_{m-1}
		&=e^{i\frac{\tau}{2}\Delta}e^{-i\tau B(\varphi_{0,m-1})}e^{i\frac{\tau}{2}\Delta}\psi_{m-1}\notag \\
		&=\Pi_{j=1}^{m} e^{i\frac{\tau}{2}\Delta}e^{-i\tau B(\varphi_{0,j-1})}e^{i\frac{\tau}{2}\Delta}\psi_{0}=\Pi_{j=1}^{m}\Gamma_{0,j-1} \psi_0,\label{z1}
	\end{align}
	where $1\leq m\leq M$, $\psi_0=\psi(\bx,0)$ and $\varphi_{0,j-1}=e^{i\frac{\tau}{2}\Delta}\psi_{j-1}$.

	\subsection{Full discrete scheme}
	
	
	Suppose the solution $\psi\in \QP(\bbR^d)$ with the corresponding projection matrix $\bP\in\mathbb P^{d\times n}$. Then we denote $K_N^n=\{\bk=(k_j)_{j=1}^n \in\bbZ^n: \, -N \leq  k_j < N\} $ for some  integer $0<N\in \bbN$ and
	$\sigma_{N}(\psi)=\{\blam = \bP\bk: \bk\in K_N^n \}.
	$
	The order of the set $\sigma_{N}(\psi)$ is $\# (\sigma_{N}(\psi))=(2N)^n$.
	
	Firstly, we discretize the tours $\bbT^n$.
	Without loss of generality, we consider a fundamental domain $[0,2\pi)^n$ and assume the discrete nodes in each dimension are
	the same, \textit{i.e.}, $N_1=N_2=\cdots=N_n=2N$. 
	The spatial discrete size $h=\pi/N$. 
	The spatial variables are evaluated on the standard
	numerical grid $\bbT^n_N$ with grid points $\by_{\bj} =(y_{1,j_1},
	y_{2,j_2},\dots, y_{n,j_n})$, $y_{1,j_1}=j_1 h$, $y_{2,j_2}=j_2 h, \dots,
	y_{n,j_n}=j_n h$, $0\leq j_1,j_2,\dots, j_n < 2N$.
	In PM, the trigonometric interpolation of the quasiperiodic function $\psi$ is
	\begin{align*}
		I_N \psi=\sum_{\blam_{\bk}\in \sigma_{N}(\psi)} \tpsi_{\bk}e^{i\blam_{\bk}\cdot \bx},
	\end{align*}
	where $\tpsi_{\bk}$ can be obtained by $n$-dimensional FFT of the parent function $\psi_p$, \textit{i.e.,}
	\begin{align*}
		\tpsi_{\bk}=\frac{1}{(2N)^n}\sum_{\by_{\bj}\in \bbT^n_N}
		\psi_p(\by_{\bj})e^{-i\bk\cdot \by_{\bj}}.
	\end{align*}
	Based on the semi-discrete scheme (\ref{z1}), we apply the PM in space and use $I_Ne^{i\frac{\tau}{2}\Delta}I_N \psi=e^{i\frac{\tau}{2}\Delta}I_N \psi$ to get the full discrete scheme 
	\begin{align}
		\Psi_{m}
		&=\calF(\tau/2)I_N \circ \calS(\tau) \circ \calF(\tau/2) I_N \Psi_{m-1}\notag\\
		&=e^{i\frac{\tau}{2}\Delta}I_Ne^{-i\tau B(\varphi_{N,0,m-1})}e^{i\frac{\tau}{2}\Delta}I_N\Psi_{m-1} \notag\\
		&=\Pi_{j=1}^{m} e^{i\frac{\tau}{2}\Delta}I_Ne^{-i\tau  B(\varphi_{N,0,j-1})}e^{i\frac{\tau}{2}\Delta}I_N\psi_{0}\notag \\
		&=\Pi_{j=1}^{m} \Upsilon_{0,j-1}\Psi_0,
		\label{eq:splitting-PM}
	\end{align}
	where $\varphi_{N,0,j-1}=e^{i\frac{\tau}{2}\Delta}I_N\Psi_{j-1}$ and $\Psi_0=I_N\psi_0$.
	
	The detailed algorithm to compute $\Psi_{m+1}$ from $\ds \Psi_{m}=\sum_{\blam\in\sigma_N(u)} \tPsi_{\blam}^m e^{i\blam\cdot \bx}$ for some $0<m<M$ contains three steps:
	
	$\bullet$ \textbf{Step 1.} For $t\in [t_m, t_m+\tau/2]$, we have
	\begin{align*}
		\phi(\bx,t_m)
		=e^{\frac{i}{2}\tau \Delta} \Psi_m
		=\sum_{\blam\in\sigma_N(\psi)} \tpsi_{\blam}^m e^{\frac{i}{2}\tau \Vert \blam\Vert^2}e^{i\blam\cdot \bx}.
	\end{align*}
	Then we denote $\tilde{\phi}_{\bk}^m:=\tpsi_{\blam}^m e^{-\frac{i}{2}\tau \Vert \blam\Vert^2}$ with $\blam=\bP\bk$.
	
	$\bullet$ \textbf{Step 2.} Applying inverse FFT yields
	\begin{align*}
		I_N\phi_p(\by_j,t_m)=\sum_{\bk\in K^n_N} \tilde{\phi}_{\bk}^m e^{i\bk\cdot\by_j},
	\end{align*}
	where the grid points $\by_j\in \bbT^n_N$.
	For $t\in [t_m, t_{m+1}]$, we have
	\begin{align*}
		\varphi_p(\by,t_{m})=e^{-i\tau (V_p+\alpha \vert \phi^*_{m,p}\vert^2 )} I_N \phi_p(\by, t_m),
	\end{align*}
	where $V_p$ is the parent function of $V$ and $\phi^*_{m,p}=I_N \phi_p(\by, t_m)$. Using FFT again, we have $\tilde{\varphi}_{\blam}^m=\langle \varphi_p(\by_j,t_{m}), e^{i \bk\cdot \by_j} \rangle_N$ with $\blam=\bP\bk$.

	$\bullet$ \textbf{Step 3.} For $t\in [t_m, t_m+\tau/2]$, similar to the Step 1, we have
	\begin{align*}
		\Psi_{m+1}= e^{\frac{i}{2}\tau \Delta} I_N \varphi(\bx, t_m)
		=\sum_{\blam\in\sigma_N(\psi)} \tilde{\varphi}_{\blam}^m
		e^{\frac{i}{2}\tau \Vert\blam \Vert^2}
		e^{i\blam\cdot \bx}.
	\end{align*}

	\subsection{Main result}
	We give the main theorem of the convergence analysis of the fully discrete scheme (\ref{eq:splitting-PM}). For simplicity, we denote $\psi(t):=\psi(\bx,t)$ in the rest of the work.
	
	\begin{thm}
		\label{thm:PMerror}
		Assume that the potential $V$ is a $\calC_{QP}^1$-smooth function with $\Vert V_p\Vert_{X_\alpha}\leq C_V$ and $\psi_p(\cdot,t)\in X_\alpha$ for $0\leq t\leq T$ and for some integer $\alpha> \max\{4,n/4\}$ with $\sup\{\Vert \psi_p(\cdot, t) \Vert_{X_{\alpha}}: 0\leq t \leq T \}\leq C_p$, then the error bound of  the fully discrete scheme \eqref{eq:splitting-PM} is	
		\begin{align*}
			\Vert \Psi_m-\psi(\cdot,t_m)\Vert\leq C(\tau^2 + N^{-\alpha}),~~0\leq m\leq M,
		\end{align*}
		where the constant $C>0$ depends on $C_V$, $C_p$, $d$, $\alpha$ and $T$.
		
	\end{thm}
	\begin{proof}
		We split the error as $\Psi_m-\psi(t_m)=\Psi_m-\psi_m + \psi_m-\psi(t_m)$, and according to the Lady Windermere's fan argument, we have
		\begin{align*}
			\psi_m-\psi(t_m)
			&=\Pi_{j=1}^{m}\Gamma_{0,j-1} \psi_0-\psi(t_m)\\
			&=\Pi_{j=2}^{m}\Gamma_{0,j-1} \Gamma_{0,0}\psi(t_0)-\Pi_{j=2}^{m}\Gamma_{1,j-1} \psi(t_1)\\
			&~~~+\Pi_{j=3}^{m}\Gamma_{1,j-1} \Gamma_{1,1}\psi(t_1)-\Pi_{j=3}^{m}\Gamma_{2,j-1} \psi(t_2)\\
			&~~~~~~~~~\vdots \\
			&~~~+\Gamma_{m-1,m-1}\psi(t_{m-1})-\psi(t_m)\\
			&=\sum_{\ell=1}^{m} \prod_{j=\ell+1}^{m} ( \Gamma_{\ell-1, j-1}\Gamma_{\ell-1,\ell-1} \psi(t_{\ell-1})-\Gamma_{\ell, j-1}\psi(t_{\ell})),
		\end{align*}
		and
		\begin{align*}
			\Psi_m-\psi_m
			&=\Pi_{j=1}^{m} \Upsilon_{0,j-1}\Psi_0-\Pi_{j=1}^{m}\Gamma_{0,j-1} \psi_0\\
			&=\Pi_{j=2}^{m} \Upsilon_{0,j-1}\Upsilon_{0,0}\psi_0
			-\Pi_{j=2}^{m} \Upsilon_{1,j-1}I_N\Gamma_{0,0}\psi_0\\
			&~~~+\Pi_{j=3}^{m} \Upsilon_{1,j-1} \Upsilon_{1,1} \psi_1-\Pi_{j=3}^{m} \Upsilon_{2,j-1}I_N\Gamma_{1,1}\psi_1\\
			&~~~~~~~~~\vdots \\
			&~~~+\Upsilon_{m-1,m-1}\psi_{m-1}- I_N\Gamma_{m-1,m-1}\psi_{m-1}+I_N\psi_m-\psi_m\\
			&=I_N\psi_m-\psi_m
			+\sum_{\ell=1}^{m} \prod_{j=\ell+1}^{m}(\Upsilon_{\ell-1,j-1} \Upsilon_{\ell-1,\ell-1}\psi_{\ell-1}-\Upsilon_{\ell,j-1} I_N\Gamma_{\ell-1,\ell-1}\psi_{\ell-1}).
		\end{align*}
		
		In order to bound the right-hand side terms of the above two equations, several auxiliary results are needed. For $\phi^* = e^{i\frac{\tau}{2}\Delta } \phi$ and $\varphi^* = e^{i\frac{\tau}{2}\Delta } \varphi$,
		denote
		\begin{align*}
			\Gamma_\phi \phi=e^{i\frac{\tau}{2}\Delta } e^{-i\tau (V+\theta \vert \phi^*\vert^2)}
			e^{i\frac{\tau}{2}\Delta }\phi,~~
			\Gamma_\varphi \varphi=e^{i\frac{\tau}{2}\Delta } e^{-i\tau (V+\theta \vert \varphi^*\vert^2)}
			e^{i\frac{\tau}{2}\Delta} \varphi.
		\end{align*}
		For $\phi_N^{**}=e^{\frac{i}{2}\tau \Delta} I_N\phi$ and $\varphi_N^{**}=e^{\frac{i}{2}\tau \Delta} I_N\varphi$,
		denote
		\begin{align*}
			\Upsilon_\phi \phi=e^{i\frac{\tau}{2}\Delta }I_N e^{-i\tau (V+\theta \vert \phi_N^{**}\vert^2)}
			e^{i\frac{\tau}{2}\Delta}I_N \phi,~~
			\Upsilon_\varphi \varphi=e^{i\frac{\tau}{2}\Delta }I_N e^{-i\tau (V+\theta \vert \varphi_N^{**}\vert^2)}
			e^{i\frac{\tau}{2}\Delta}I_N \varphi.
		\end{align*} 
		Then, in order to estimate $\Vert \psi_m-\psi(t_m)\Vert $, we need the upper bounds for $\Vert \Gamma_\phi \phi- \Gamma_\psi \psi\Vert$ (Theorem \ref{thm:time-error-part} (i)) and $ \Vert \Gamma_{\ell-1,\ell-1} \psi(t_{\ell-1})-\psi(t_{\ell})\Vert $ (Theorem \ref{thm:time-error-part} (ii)). To estimate $\Vert \Psi_m-\psi_m\Vert $, 
		we need the upper bounds for $\Vert \Upsilon_\phi \phi- \Upsilon_\psi \psi\Vert$ (Theorem \ref{thm:phierror}) and $ \Vert \Upsilon_{\ell-1,\ell-1}\psi_{\ell-1} -I_N\Gamma_{\ell-1,\ell-1}\psi_{\ell-1}\Vert$ (Theorem \ref{thm:space-error-part}). Furthermore, the analysis of the operator splitting method requires estimates of intermediate solutions, which are given in  Lemma \ref{lem:eq-inter-values}. With the help of these auxiliary results, we have
		\begin{align*}
			\Vert \Psi_m- \psi(\cdot,t_m)\Vert
			&\leq \Vert \psi_m-\psi(t_m)\Vert+ \Vert  \Psi-\psi_m \Vert\\
			&\hspace{-0.9in}\leq   \sum_{\ell=1}^{m} \Big \Vert \prod_{j=\ell+1}^{m} ( \Gamma_{\ell-1, j-1}\Gamma_{\ell-1,\ell-1} \psi(t_{\ell-1})-\Gamma_{\ell, j-1}\psi(t_{\ell})) \Big \Vert\\
			&~~\hspace{-0.8in}+ \Big \Vert 
			I_N\psi_m-\psi_m \Big \Vert 
			+\sum_{\ell=1}^{m} \Big \Vert  \prod_{j=\ell+1}^{m}(\Upsilon_{\ell-1,j-1} \Upsilon_{\ell-1,\ell-1}\psi_{\ell-1}-\Upsilon_{\ell,j-1} I_N\Gamma_{\ell-1,\ell-1}\psi_{\ell-1}) \Big \Vert\\
			&\hspace{-0.9in}\leq \sum_{\ell=1}^{m} e^{C(C_V+\vert \theta \vert +C_p^2)(m-\ell)\tau}\Vert \Upsilon_{\ell-1,\ell-1}\psi_{\ell-1}-\psi(t_\ell)\Vert
			~~(\mbox{Theorem \ref{thm:time-error-part} (i)})\\
			&~~\hspace{-0.8in}+C N^{-\alpha} \vert\psi_{m,p}\vert_{X_{\alpha}}+\sum_{\ell=1}^{m} e^{C(C_V+\vert \theta \vert C_p^2)(m-\ell)\tau}\Vert \Upsilon_{\ell-1,\ell-1}\psi_{\ell-1}-I_N\Gamma_{\ell-1,\ell-1}\psi_{\ell-1}\Vert~
			(\mbox{Theorem \ref{thm:phierror}})\\
			&\hspace{-0.9in}\leq C\sum_{\ell=1}^{m} e^{C(C_V+\vert \theta \vert C_p^2)(m-\ell)\tau}\tau^3 ~~(\mbox{Theorem \ref{thm:time-error-part} (ii)})\\
			&~~\hspace{-0.8in}+C N^{-\alpha} \vert \psi_{m,p}\vert_{X_{\alpha}}+ C \sum_{\ell=1}^{m} e^{C(C_V+\vert \theta \vert C_p^2)(m-\ell)\tau}\tau N^{-\alpha} \vert \psi_{\ell-1,p}\vert_{X_{\alpha}}
			~(\mbox{Theorem \ref{thm:space-error-part}})\\
			&\hspace{-0.9in}\leq C(\tau^2+N^{-\alpha})~~(\mbox{Note that $m\tau\leq T$}),
		\end{align*}
		which completes the proof.
	\end{proof}


	\section{Auxiliary estimates}
	\label{sec:analysis}
	
	\subsection{Estimates in time}

	The follow lemma introduces some bounds related to the nonlinear operator $B(\varphi)$.
	
	\begin{lemma}
		\label{lemma:B-pro}
		For $\phi,\varphi,\psi\in \QP(\bbR^d)$ and corresponding parent functions $\phi_p, \varphi_p$ and $\psi_p$, respectively, the following bounds hold for $\alpha>n/4$:
		
		(i) $\Vert B(\phi) \varphi\Vert  \leq C(\Vert V_p \Vert_{X_{\alpha}}
		+\vert \theta\vert \Vert \phi_p\Vert^2_{X_{\alpha}} ) \Vert \varphi\Vert$;
		
		(ii) $\Vert (B(\phi)- B(\psi))\varphi\Vert 
		\leq C\vert \theta\vert (\Vert \phi_p \Vert_{X_{\alpha}}
		+ \Vert \psi_p\Vert_{X_{\alpha}}) \Vert \varphi_p\Vert_{X_{\alpha}}
		\Vert \phi-\psi\Vert$.
	\end{lemma}
	
	\begin{proof}
		(i) We apply the following norm inequality \cite[Lemma 1 (iii)]{Thalhammer2012Convergence}
		\begin{align*}
			\Vert \phi_p \varphi_p \Vert_{X_{\alpha}}\leq C \Vert \phi_p\Vert_{X_{\alpha}}
			\Vert \varphi_p\Vert_{X_{\alpha}},~~ \phi_p, \varphi_p\in X_{\alpha}(\bbT^n),
		\end{align*} 
		the Proposition \ref{pro:projectmatrix} and the inequality \eqref{eq:inftyembedX} to get
		\begin{align*}
			\Vert B(\phi) \varphi\Vert
			=\Vert (V+\theta \vert \phi\vert^2)\varphi\Vert
			\leq C (\Vert V_p\Vert_{X_{\alpha}}+\vert \theta\vert \Vert \phi_p\Vert^2_{X_{\alpha}}) \Vert \varphi\Vert.
		\end{align*}
		
		(ii) Since 
		$(B(\phi)- B(\psi))\varphi
		=\theta (\vert \phi\vert^2-\vert \psi\vert^2)\varphi
		=\theta ((\phi- \psi)\bar{\phi}+(\overline{\phi-\psi})\psi)\varphi,$
		then we obtain
		\begin{align*}
			\Vert (B(\phi)- B(\psi))\varphi \Vert
			&\leq C \vert \theta\vert \Vert (\phi- \psi)\bar{\phi}+(\overline{\phi-\psi})\psi\Vert \cdot
			\Vert \varphi_p\Vert_{X_{\alpha}}\\
			&\leq C \vert \theta\vert (\Vert (\phi- \psi) \bar{\phi}\Vert
			+\Vert (\overline{\phi-\psi})  \psi\Vert) 
			\Vert \varphi_p\Vert_{X_{\alpha}}\\
			&\leq C \vert \theta\vert (\Vert \phi- \psi\Vert\cdot  \Vert \phi_p\Vert_{X_{\alpha}}
			+\Vert \phi-\psi\Vert \cdot \Vert \psi_p\Vert_{X_{\alpha}}) 
			\Vert \varphi_p\Vert_{X_{\alpha}}\\
			&=C \vert \theta\vert ( \Vert \phi_p\Vert_{X_{\alpha}}+ \Vert \psi_p\Vert_{X_{\alpha}}) 
			\Vert \varphi_p\Vert_{X_{\alpha}}\Vert \phi- \psi\Vert.
		\end{align*}
		
	\end{proof}

	For the exponential operators, we have the following estimates.
	\begin{lemma}
		\label{lem:ABerror}
		For $\phi, \varphi\in\QP(\bbR^d)$ with the corresponding parent functions $\phi_p$ and $\varphi_p$, respectively, we have (i)
		$
		\Vert e^{it\Delta} \phi\Vert=\Vert \phi\Vert.
		$
		(ii) If further $\Vert V_p\Vert_{X_{\alpha}}\leq C_V$ and 
		$\Vert \phi_p\Vert_{X_{\alpha}},\Vert \varphi_p\Vert_{X_{\alpha}}\leq C_p$ with $\alpha > n/4$, then
		\begin{align*}
			\Vert e^{-itB(\phi)}\phi-e^{-itB(\varphi)}\varphi\Vert
			\leq e^{C(C_V+\vert \theta\vert C_p^2)t} \Vert \phi-\varphi \Vert.
		\end{align*}
	\end{lemma}

	\begin{proof}
		(i) Since the operator $\Delta$ is self-adjoint in the $L^2_{QP}$ inner product  (see Lemma \ref{thm:selfadjoint}), then according to the Stone's theorem \cite{Stone1930linear}, the conclusion (i) holds.
		
		(ii) Consider the following two initial value problems
		\begin{align}
			\begin{cases}
				i\dfrac{d}{dt} \psi_1(t)=B(\phi) \psi_1(t),\\
				\psi_1(0)=\phi,
			\end{cases}
			\label{initialproblem-1}
		\end{align}
		and 
		\begin{align}
			\begin{cases}
				i\dfrac{d}{dt} \psi_2(t)=B(\varphi) \psi_2(t),\\
				\psi_2(0)=\varphi,
			\end{cases}
			\label{initialproblem-2}
		\end{align}
		whose analytical solutions are $\psi_1=e^{-itB(\phi)}\phi$ and $\psi_2=e^{-itB(\varphi)}\varphi.$
		Furthermore, we consider the initial value problem
		\begin{align*}
			\begin{cases}
				i\dfrac{d}{dt} (\psi_1-\psi_2)(t)=B(\phi) \psi_1(t)-B(\varphi) \psi_2(t),\\
				(\psi_1-\psi_2)(0)=\phi-\varphi.
			\end{cases}
		\end{align*}
		Since $B(\phi) \psi_1(t)-B(\varphi) \psi_2(t)=B(\phi) (\psi_1(t)-\psi_2(t))+(B(\phi)-B(\varphi)) \psi_2(t)$, we apply the  variation-of-constants formula to obtain
		\begin{align*}
			(\psi_1-\psi_2)(t)
			&=e^{-itB(\phi)}(\phi-\varphi)+\int_{0}^{t} e^{-i(t-\tau)B(\phi)}(B(\phi)-B(\varphi)) \psi_2(\tau)d\tau.
		\end{align*}
		By $\Vert e^{-itB(\phi)}\phi \Vert= \Vert \phi \Vert $,
		$\psi_{2,p}=e^{-i\tau (V_p+\theta \vert \varphi_p\vert^2)}\varphi_p$ and
		Lemma 2 in \cite{Thalhammer2012Convergence}, which gives 
		\begin{align*}
			\Vert e^{-i\tau (V_p+\theta \vert \phi_p\vert^2)}\varphi_p \Vert_{X_{\alpha}}
			\leq e^{C(\Vert V_p\Vert_{X_{\alpha}}+\vert \theta\vert 
				\Vert \phi_p\Vert^2_{X_{\alpha}} )\tau}  \Vert \varphi_p\Vert_{X_{\alpha}},
		\end{align*}
		we have
		\begin{align*}
			&\Big \Vert \int_{0}^{t} e^{-i(t-\tau)B(\phi)}(B(\phi)-B(\varphi)) \psi_2(\tau) d\tau \Big \Vert\\
			&\quad\leq  \int_{0}^{t} \Vert e^{-i(t-\tau)B(\phi)}(B(\phi)-B(\varphi)) \psi_2(\tau)\Vert d\tau\\
			&\quad\leq \int_{0}^{t} 
			\Vert(B(\phi)-B(\varphi)) \psi_2(\tau)\Vert d\tau\\
			&\quad\leq C \vert \theta \vert 
			(\Vert  \phi_p\Vert_{X_{\alpha}} + \Vert \varphi_p \Vert_{X_{\alpha}})
			\int_{0}^{t} 
			\Vert \phi-\varphi \Vert \cdot  \Vert \psi_{2,p}\Vert_{X_{\alpha}} d\tau~~(\mbox{Lemma \ref{lemma:B-pro} (ii)})\\
			&\quad\leq  C \vert \theta \vert 
			(\Vert  \phi_p\Vert_{X_{\alpha}} + \Vert \varphi_p \Vert_{X_{\alpha}})
			\int_{0}^{t}
			\Vert \phi-\varphi \Vert e^{C(\Vert V_p\Vert_{X_{\alpha}}+\vert \theta\vert 
				\Vert \varphi_p\Vert^2_{X_{\alpha}} )\tau}  \Vert \varphi_p\Vert_{X_{\alpha}} d\tau\\
			&\quad\leq C \vert \theta \vert 
			(\Vert  \phi_p\Vert_{X_{\alpha}} + \Vert \varphi_p \Vert_{X_{\alpha}})
			\Vert \varphi_p\Vert_{X_{\alpha}}
			(e^{C(\Vert V_p\Vert_{X_{\alpha}}+\vert \theta\vert 
				\Vert \varphi_p\Vert^2_{X_{\alpha}} )t}-1)
			\Vert \phi-\varphi \Vert\\
			&\quad\leq 2 C \vert \theta \vert
			e^{C( C_V+\vert \theta\vert C_p^2)t}
			\Vert \phi-\varphi \Vert.
		\end{align*}
		Consequently, we apply  $1+x\leq e^x$ for $x\geq 0$ to get
		\begin{align*}
			\Vert \psi_1-\psi_2\Vert 
			&\leq \Vert e^{-itB(\phi)}(\phi-\varphi)\Vert + \Big \Vert \int_{0}^{t} e^{-i(t-\tau)B(\phi)}(B(\phi)-B(\varphi)) \psi_2(\tau)d\tau \Big\Vert \\
			&\leq (1+2C \vert \theta \vert e^{C( C_V+\vert \theta\vert C_p^2)t})
			\Vert \phi-\varphi \Vert\\
			&\leq  e^{C( C_V+\vert \theta\vert C_p^2)t}
			\Vert \phi-\varphi \Vert,
		\end{align*}
		which completes the proof.
	\end{proof}
	



	Then we show the norm-preserving property of parent function under the operation of  $e^{i\frac{\tau}{2}\Delta }$.
	
	\begin{lemma}
		\label{lem:vp-norm}
		For $\phi\in\QP(\bbR^d)$ with
		\begin{align}
			\phi = \sum_{\blam_j\in \sigma(\phi)} \hat{\phi}_j e^{i\blam_j \cdot \bx},
			\label{eq:Fourierseries-v}
		\end{align}
		and $\phi_p\in X_{\alpha}$ for some $\alpha \geq 0$, we have 
		$\Vert \phi^*_p \Vert_{X_{\alpha}}= \Vert \phi_p \Vert_{X_{\alpha}}$ where $\phi^*=e^{i\frac{\tau}{2}\Delta } \phi$.
	\end{lemma}
	
	\begin{proof}
		Due to $\phi_p = \sum_{\bk_j\in \bK(\phi)} \hat{\phi}_j e^{i\bk_j \cdot \by} ~(\by\in \bbT^n)$ and 
		$\Delta^{\alpha} (e^{i\blam_j \cdot \bx})= \Vert \blam_j\Vert^{2\alpha} e^{i\blam_j \cdot \bx}$, we obtain
		\begin{align*}
			e^{i\frac{\tau}{2}\Delta } e^{i\blam_j \cdot \bx}= e^{i\frac{\tau}{2} \Vert \blam_j \Vert^2 } e^{i\blam_j \cdot \bx},
		\end{align*}
		which implies
		\begin{align*}
			\phi^* =\sum_{\blam_j\in \sigma(\phi)} \hat{\phi}_j e^{i \frac{\tau}{2} \Vert \blam_j \Vert^2} e^{i\blam_j \cdot \bx},~~\phi^*_p= \sum_ {\bk_j\in \bK(\phi)}\hat{\phi}_j e^{i \frac{\tau}{2} \Vert \bP\bk_j \Vert^2} e^{i\bk_j \cdot \by}.
		\end{align*}
		Hence, 
		$\Vert \phi^*_p \Vert_{X_{\alpha}}=\Vert \phi^*_p \Vert+\Vert  \Delta^{\alpha} \phi^*_p \Vert
		=\Vert \phi_p \Vert+\Vert  \Delta^{\alpha} \phi_p \Vert=\Vert \phi_p \Vert_{X_{\alpha}}.$
	\end{proof}
	
	To give an upper bound estimate for $\Vert \psi_m-\psi(t_m)\Vert $, we analyze the upper bounds for $\Vert \Gamma_\phi \phi-\Gamma_\varphi  \varphi\Vert$ and 
	$\Vert \Gamma_{\ell-1,\ell-1} \psi(t_{\ell-1})-\psi(t_\ell)\Vert$.
	
	\begin{thm}
		\label{thm:time-error-part}
		If $\phi, \varphi, \psi(t_{\ell-1})\in \QP(\bbR^d)$ satisfy 
		$\Vert \phi_p \Vert_{X_{\alpha}}, \Vert \varphi_p \Vert_{X_{\alpha}},  \Vert \psi_p(t_{\ell-1}) \Vert_{X_{\alpha}}\leq C$ for $\ell=1,2,\cdots,m$ and $\alpha>\max\{4,n/4\}$.
		Then, the following estimates hold:
		
		(i) $\Vert \Gamma_\phi \phi-\Gamma_\varphi \varphi\Vert
		\leq e^{C( C_V+\vert \theta\vert C_p^2)\tau}
		\Vert \phi- \varphi \Vert.$
		
		(ii) $\Vert \Gamma_{\ell-1,\ell-1} \psi(t_{\ell-1})-\psi(t_\ell)\Vert \leq C\tau^3.$ 
	\end{thm}

	\begin{proof}

		
		(i)
		Applying Lemma \ref{lem:ABerror} and Lemma \ref{lem:vp-norm}, we can obtain
		\begin{align*}
			\Vert \Gamma_\phi \phi-\Gamma_\varphi  \varphi\Vert
			&=\Vert e^{i\frac{\tau}{2}\Delta } 
			e^{-i\tau B(\phi^*)}
			e^{i\frac{\tau}{2}\Delta } \phi-
			e^{i\frac{\tau}{2}\Delta } 
			e^{-i\tau B(\varphi^*)}
			e^{i\frac{\tau}{2}\Delta }  \varphi \Vert\\
			&=\Vert e^{-i\tau B(\phi^*)}
			e^{i\frac{\tau}{2}\Delta } \phi-
			e^{-i\tau B(\varphi^*)}
			e^{i\frac{\tau}{2}\Delta }  \varphi \Vert\\
			&\leq e^{C( C_V+\vert \theta\vert C_p^2)\tau}
			\Vert \phi- \varphi \Vert.
		\end{align*}
		
		(ii)
		For an operator $D$, we have
		\begin{align*}
			e^{tD}=I+e^{\tau D}\arrowvert^t_{\tau =0}
			=I+\int_{0}^{t} \frac{d}{d\tau} e^{\tau D} d\tau
			=I+\int_{0}^{t} e^{\tau D} D d\tau,
		\end{align*}
		which implies the expansion
		\begin{align*}
			e^{tD}=I+t D+\frac{1}{2} t^2 D^2 +\int_{0}^{t} \int_{0}^{\tau_1} \int_{0}^{\tau_2}
			e^{\tau_3 D} D^3 d\tau_3 d\tau_2 d\tau_1,
			~~0<\tau_3\leq \tau_2 \leq \tau_1\leq t.
		\end{align*}
		Therefore, denoting $\psi^*_{\ell-1}=e^{i\frac{\tau}{2}\Delta }\psi(t_{\ell-1})$, we have
		\begin{align*}
			e^{-i\tau B(\psi^*_{\ell-1})}\psi^*_{\ell-1}
			=\psi^*_{\ell-1} -i\tau B(\psi^*_{\ell-1})\psi^*_{\ell-1}
			-\frac{1}{2} \tau^2 B^2(\psi^*_{\ell-1})\psi^*_{\ell-1} 
			+r_1(\tau),
		\end{align*}
		where
		\begin{align*}
			r_1(\tau)=i \int_{0}^{\tau} \int_{0}^{\tau_1} \int_{0}^{\tau_2}
			e^{-i\tau_3 B(\psi^*_{\ell-1})} B^3(\psi^*_{\ell-1}) \psi^*_{\ell-1} d\tau_3 d\tau_2 d\tau_1.
		\end{align*}
		Note that for the quasiperiodic function $\phi$, we apply Lemma \ref{lemma:Birkhoff} to get
		$\Vert \phi \Vert=\Vert \phi_p\Vert=\Vert \phi_p\Vert_{X_0}
		\leq \Vert \phi_p\Vert_{X_{\alpha}}$ for $\alpha\geq 0.$
		Hence, it follows that
		\begin{align*}
			\Vert r_1(\tau)\Vert
			&\leq \int_{0}^{\tau} \int_{0}^{\tau_1} \int_{0}^{\tau_2}
			\Vert e^{-i\tau_3 B(\psi^*_{\ell-1})} B^3(\psi^*_{\ell-1}) \psi^*_{\ell-1} \Vert  d\tau_3 d\tau_2 d\tau_1\\
			&\leq \int_{0}^{\tau} \int_{0}^{\tau_1} \int_{0}^{\tau_2}
			\Vert B^3(\psi^*_{\ell-1}) \psi^*_{\ell-1} \Vert  d\tau_3 d\tau_2 d\tau_1\\
			&\leq \int_{0}^{\tau} \int_{0}^{\tau_1} \int_{0}^{\tau_2}
			C(\Vert V_p\Vert_{X_{\alpha}}, \Vert \psi_p(t_{\ell-1})\Vert_{X_{\alpha}})
			d\tau_3 d\tau_2 d\tau_1\leq C\tau^3.
		\end{align*}
		Furthermore, we obtain
		\begin{align}
			&e^{i\frac{\tau}{2} \Delta} e^{-i\tau B(\psi^*_{\ell-1})} \psi^*_{\ell-1}\nonumber\\
			&=e^{i\frac{\tau}{2} \Delta} \psi^*_{\ell-1} 
			-i\tau e^{i\frac{\tau}{2} \Delta} B(\psi^*_{\ell-1}) \psi^*_{\ell-1}
			-\frac{1}{2} \tau^2 e^{i\frac{\tau}{2} \Delta}
			B^2(\psi^*_{\ell-1})\psi^*_{\ell-1} 
			+e^{i\frac{\tau}{2} \Delta}r_1(\tau)\nonumber\\
			&=e^{i\frac{\tau}{2} \Delta} \psi^*_{\ell-1} 
			-i\tau e^{i\frac{\tau}{2} \Delta} B(\psi^*_{\ell-1}) \psi^*_{\ell-1}
			+e^{i\frac{\tau}{2} \Delta}r_1(\tau)\label{z3}\\
			&~~~-\frac{1}{2} \tau^2 B^2(\psi^*_{\ell-1})\psi^*_{\ell-1}
			- \frac{i}{4} \tau^2 \int^\tau_{0} e^{i\frac{\tau_1}{2} \Delta}\Delta
			B^2(\psi^*_{\ell-1}) \psi^*_{\ell-1} d\tau_1\nonumber\\
			&=e^{i\frac{\tau}{2} \Delta} \psi^*_{\ell-1} 
			-i\tau e^{i\frac{\tau}{2} \Delta} B(\psi^*_{\ell-1})\psi^*_{\ell-1}
			-\frac{1}{2} \tau^2 B^2(\psi^*_{\ell-1})\psi^*_{\ell-1}
			+r_2(\tau),\nonumber
		\end{align}
		where
		\begin{align*}
			r_2(\tau)=e^{i\frac{\tau}{2} \Delta}r_1(\tau)
			- \frac{i}{4} \tau^2 \int^\tau_{0} e^{i\frac{\tau_1}{2} \Delta}\Delta
			B^2(\psi^*_{\ell-1}) \psi^*_{\ell-1} d\tau_1.
		\end{align*}
		Thus $r_2(\tau)$ contains the residual function $r_1(\tau)$.
		Let $\varphi^{**}_{\ell-1}=(V+\theta \vert \psi^*_{\ell-1} \vert^2)^2 \psi^*_{\ell-1}$, then $\varphi^{**}_{\ell-1,p}=(V_p+\theta \vert \psi^*_{\ell-1,p} \vert^2)^2 \psi^*_{\ell-1,p}$ is its parent function and 
		\begin{align*}
			\Vert r_2(\tau)\Vert 
			&\leq \Vert e^{i\frac{\tau}{2} \Delta}r_1(\tau) \Vert
			+\frac{1}{4} \tau^2   \int^\tau_{0} \Vert  e^{i\frac{\tau_1}{2} \Delta}
			\Delta B^2(\psi^*_{\ell-1}) \psi^*_{\ell-1} \Vert d\tau_1\\
			&=\Vert r_1(\tau) \Vert
			+\frac{1}{4} \tau^2   \int^\tau_{0} \Vert 
			\Delta (V+\theta \vert \psi^*_{\ell-1} \vert^2)^2 \psi^*_{\ell-1} \Vert d\tau_1\\
			&\leq \Vert r_1(\tau) \Vert
			+\frac{1}{4} \tau^2  \sigma_{max}^2(\bP)\int^\tau_{0} \Vert 
			\Delta \varphi^{**}_{\ell-1,p} \Vert d\tau_1 
			~~ \mbox{(Proposition \ref{pro:bound-deltav})}\\
			&\leq \Vert r_1(\tau) \Vert
			+\frac{1}{4} \tau^2  \sigma_{max}^2(\bP)C(\Vert V_p \Vert_{X_{\alpha}},\Vert \psi_p(t_{\ell-1}) \Vert_{X_{\alpha}})\int^\tau_{0} d\tau_1\leq C \tau^3,
		\end{align*}
		where $C$ depends on $\sigma_{max}(\bP)$, and the penultimate inequality holds due to the following relation for $\alpha \geq 1$
		\begin{align*}
			\Vert  \Delta \varphi^{**}_{\ell-1,p} \Vert
			\leq \Vert \varphi^{**}_{j-1,p} \Vert_{X_1}
			&\leq \Vert (V_p+\theta \vert \psi^*_{\ell-1,p} \vert^2)^2\Vert_{X_1} \Vert \psi_p(t_{\ell-1}) \Vert_{X_1}\\
			&\leq \Vert (V_p+\theta \vert \psi^*_{\ell-1,p} \vert^2)^2\Vert_{X_\alpha} \Vert \psi_p(t_{\ell-1}) \Vert_{X_\alpha}\\
			&\leq C(\Vert V_p \Vert_{X_{\alpha}},\Vert \psi_p(t_{\ell-1}) \Vert_{X_{\alpha}}).
		\end{align*}
		
		Now we are in the position to expand $\psi(t_\ell)$ for the sake of analysis.
		By the Duhamel formula of the evaluation equation, we have
		\begin{align*}
			\psi(t_\ell)
			&=e^{i\tau \Delta}\psi(t_{\ell-1})
			+\int_{0}^\tau e^{i(\tau-\tau_1)\Delta} (V+\theta \vert \psi(t_{\ell-1}+\tau_1) \vert^2)\psi(t_{\ell-1}+\tau_1) d\tau_1,~0<\tau_1\leq \tau.
		\end{align*}
		Denote $R_{\ell-1}(\tau, \tau_1)=e^{i(\tau-\tau_1)\Delta} (V+\theta \vert \psi(t_{\ell-1}+\tau_1) \vert^2)\psi(t_{\ell-1}+\tau_1)$.
		Similarly, it follows that
		\begin{align*}
			\psi(t_{\ell-1}+\tau_1)
			&=e^{i\tau_1 \Delta}\psi(t_{\ell-1})+\int_{0}^{\tau_1} R_{\ell-1}(\tau_1,\tau_2)d\tau_2\\
			&=\psi^*_{\ell-1}(\tau_1)+\int_{0}^{\tau_1} R_{\ell-1}(\tau_1,\tau_2)d\tau_2,
		\end{align*}
		where $0<\tau_2\leq\tau_1$.
		Therefore, 
		\begin{align*}
			&\psi(t_\ell)=e^{i\tau \Delta}\psi(t_{\ell-1})\\
			&+\int_{0}^\tau e^{i(\tau-\tau_1)\Delta} \Big (V+\theta 
			\Big \vert \psi^*_{\ell-1}(\tau_1)+\int_{0}^{\tau_1} R_{\ell-1}(\tau_1,\tau_2)d\tau_2 \Big \vert^2\Big ) \Big (\psi^*_{\ell-1}(\tau_1)+\int_{0}^{\tau_1} R_{\ell-1}(\tau_1,\tau_2)d\tau_2 \Big ) d\tau_1\\
			&=e^{i\tau \Delta}\psi(t_{\ell-1})
			+\int_{0}^\tau e^{i(\tau-\tau_1)\Delta} (V+\theta 
			\vert \psi^*_{\ell-1}(\tau_1)\vert^2) \psi^*_{\ell-1}(\tau_1)d\tau_1\\
			&+\int_{0}^\tau e^{i(\tau-\tau_1)\Delta} \int_{0}^{\tau_1} [(V+2\theta 
			\vert \psi^*_{\ell-1}(\tau_1)\vert^2) R_{\ell-1}(\tau_1,\tau_2)
			+ \theta (\psi^*_{\ell-1}(\tau_1))^2 \overline{R}_{\ell-1}(\tau_1,\tau_2)] d\tau_2 d\tau_1+r_3(\tau),
		\end{align*}
		where
		\begin{align*}
			r_3(\tau)&=\int_{0}^\tau e^{i(\tau-\tau_1)\Delta}
			\Big \{
			\theta \Big \vert \int_{0}^{\tau_1} R_{\ell-1}(\tau_1,\tau_2)d\tau_2 \Big\vert^2  \psi^*_{\ell-1}(\tau_1)
			+ 2\theta \Big ( \int_{0}^{\tau_1} R_{\ell-1}(\tau_1,\tau_2)d\tau_2 \Big)^2  \overline{\psi^*_{\ell-1}}(\tau_1)
			\Big \} d\tau_1\\
			&~~~+\Big \vert \int_{0}^{\tau_1} R_{\ell-1}(\tau_1,\tau_2)d\tau_2 \Big\vert^2\int_{0}^{\tau_1} R_{\ell-1}(\tau_1,\tau_2)d\tau_2.
		\end{align*}
		Similar to the estimate of $r_2(\tau)$, we could bound $r_3(\tau)$ as $\Vert r_3(\tau) \Vert \leq C\tau^3$.
		
		Applying the Duhamel formula again, we have
		\begin{align*}
			\psi(t_{\ell-1}+\tau_2)
			= \psi^*_{\ell-1}(\tau_2)+ \int_{0}^{\tau_2} R_{\ell-1}(\tau_2,\tau_3) d\tau_3,~~0<\tau_3\leq\tau_2,
		\end{align*} 
		and 
		\begin{align*}
			\psi^*_{\ell-1}(\tau_2)=e^{i(\tau_2-\tau_1)\Delta}\psi^*_{\ell-1}(\tau_1)
			=(I+R_{\Delta}(\tau_1,\tau_2))\psi^*_{\ell-1}(\tau_1),
		\end{align*}
		where 
		\begin{align*}
			R_{\Delta}(\tau_1,\tau_2)=\int_{0}^{\tau_1} e^{it_1 \Delta}\Delta dt_1
			+\int_{0}^{\tau_2} e^{-it_2 \Delta}\Delta dt_2
			+\int_{0}^{\tau_1}\int_{0}^{\tau_2} (e^{it_1 \Delta}-e^{-it_2 \Delta}) \Delta d t_1 d t_2.
		\end{align*}
		Therefore, 
		\begin{align*}
			R_{\ell-1}(\tau_1,\tau_2)
			&=(I+\overline{R}_{\Delta}) (V+\theta \vert \psi(t_{\ell-1}+\tau_2) \vert^2) 
			\psi(t_{\ell-1}+\tau_2) \\
			&=(V+\theta \vert \psi^*_{\ell-1}(\tau_1) \vert^2) \psi^*_{\ell-1}(\tau_1)+r_4(\tau_1,\tau_2),
		\end{align*}
		where $\overline{R}_{\Delta}=R^{-1}_{\Delta}$ and each of the terms in $r_4(\tau_1,\tau_2)$ contains the coupling of two integral terms.
		Now, we can obtain
		\begin{align*}
			\psi(t_\ell)&=e^{i\tau \Delta}\psi(t_{\ell-1})
			+\int_{0}^\tau e^{i(\tau-\tau_1)\Delta} (V+\theta 
			\vert \psi^*_{\ell-1}(\tau_1)\vert^2) \psi^*_{\ell-1}(\tau_1)d\tau_1+r_3(\tau)\\
			&~~+\int_{0}^\tau e^{i(\tau-\tau_1)\Delta} \int_{0}^{\tau_1} [(V+2\theta 
			\vert \psi^*_{\ell-1}(\tau_1)\vert^2) R_{\ell-1}(\tau_1,\tau_2)
			+ \theta (\psi^*_{\ell-1}(\tau_1))^2 \overline{R}_{\ell-1}(\tau_1,\tau_2)] d\tau_2 d\tau_1
			\\
			&=e^{i\tau \Delta}\psi(t_{\ell-1})
			+\int_{0}^\tau e^{i(\tau-\tau_1)\Delta} B(\psi^*_{\ell-1}(\tau_1)) \psi^*_{\ell-1}(\tau_1)d\tau_1\\
			&~~+\int_{0}^\tau e^{i(\tau-\tau_1)\Delta} \tau_1 B^2(\psi^*_{\ell-1}(\tau_1)) \psi^*_{\ell-1}(\tau_1)d\tau_1
			+r_5(\tau)
			+r_3(\tau),
		\end{align*}
		where $r_5(\tau)$ contains the higher order residual function $r_4(\tau_1,\tau_2)$ and $\Vert r_5(\tau) \Vert\leq C\tau^3$. We subtract this equation from (\ref{z3}) and apply the trapezoidal rule to get
		\begin{align*}
			&\psi(t_\ell)-e^{i\frac{\tau}{2} \Delta} e^{-i\tau B(\psi^*_{\ell-1})}e^{i\frac{\tau}{2} \Delta} \psi(t_{\ell-1})\\
			&=\int_{0}^\tau e^{i(\tau-\tau_1)\Delta} B(\psi^*_{\ell-1}(\tau_1)) \psi^*_{\ell-1}(\tau_1)d\tau_1\\
			&~~+\int_{0}^\tau e^{i(\tau-\tau_1)\Delta} \tau_1 B^2(\psi^*_{\ell-1}(\tau_1)) \psi^*_{\ell-1}(\tau_1)d\tau_1
			+r_5(\tau)+r_3(\tau)\\
			&~~-i\tau e^{i\frac{\tau}{2} \Delta} B(\psi^*_{\ell-1})\psi^*_{\ell-1}
			-\frac{1}{2} \tau^2 B^2(\psi^*_{\ell-1})\psi^*_{\ell-1}
			+r_2(\tau)\\
			&=\frac{\tau}{2}(f(0)+f(\tau))-\int_{0}^{\tau} f'(\tau_1) d\tau_1 +r_2(\tau)+r_3(\tau)+r_5(\tau)\\
			&=\frac{\tau^3}{2}\int_{0}^{1} \theta(1-\theta) f''(\theta) d\theta +r_2(\tau)+r_3(\tau)+r_5(\tau),
		\end{align*}
		where
		$f(\tau)=-ie^{i\frac{\tau}{2} \Delta} B(\psi^*_{\ell-1})\psi^*_{\ell-1}-\frac{1}{2} \tau B^2(\psi^*_{\ell-1})\psi^*_{\ell-1}$ and $\Vert f_p''(\tau) \Vert\leq C$ depending on the $H^4$-norm of $V$, $\psi^*_{\ell-1}$.
		Then, it follows that
		\begin{align*}
			&\Vert \psi(t_\ell)-e^{i\frac{\tau}{2} \Delta} e^{-i\tau B(\psi^*_{\ell-1})}e^{i\frac{\tau}{2} \Delta} \psi(t_{\ell-1})\Vert\\ 
			&~~\leq \frac{\tau^3}{2}\int_{0}^{1} \Vert f''(\theta) \Vert d\theta 
			+\Vert r_2(\tau)\Vert + \Vert r_3(\tau)\Vert + \Vert r_5(\tau)\Vert\\
			&~~\leq C\tau^3,
		\end{align*}
		which proves (ii).
	\end{proof}
	
	\subsection{Estimates in space}
	We first refer the interpolation error estimate of the quasiperiodic function \cite{Jiang2024Numerical}.
	\begin{lemma}
		\label{thm:pm error-2}
		Suppose that $\psi(\bx)\in \QP(\bbR^d)$ and its parent function $\psi_p(\by)\in X_{\alpha}(\bbT^n)$ with $\alpha > n/2$. There exists a constant $C$, independent of $\psi_p$ and $N$ such that
		\begin{align*}
			\Vert I_N \psi-\psi\Vert \leq CN^{-\alpha}\vert \psi_p \vert_{X_{\alpha}}.
		\end{align*}
	\end{lemma}
	
	\begin{lemma}
		\label{lem:trunAB}
		The following relations hold:
		
		(i) For $\phi\in \QP(\bbR^d)$ with the Fourier series expansion \eqref{eq:Fourierseries-v}, it holds that
		\begin{align*}
			\Vert I_N e^{i\frac{\tau}{2}\Delta} I_N \phi\Vert=\Vert I_N \phi\Vert.
		\end{align*}
		
		(ii) For $\phi,\varphi\in\QP(\bbR^d)$ with $n$-dimensional parent functions $\phi_p,\,\varphi_p\in X_{\alpha}(\bbT^n)$ for $\alpha>n/4$, if there exist constants $C_V>0$ and $C_p>0$ such that $\Vert V_p\Vert_{X_{\alpha}}\leq C_V$ and $\Vert \phi_p\Vert_{X_{\alpha}}, \Vert \varphi_p\Vert_{X_{\alpha}}\leq C_p$, then
		\begin{align*}
			\Vert I_N(e^{-i\tau B(\phi)}\phi-e^{-i\tau B(\varphi)}\varphi)\Vert
			\leq e^{C(C_V+\vert \theta\vert C_p^2)\tau} \Vert I_N(\phi-\varphi)\Vert.
		\end{align*}
	\end{lemma}
	
	\begin{proof}
		(i) Since
		\begin{align*}
			I_N \phi=\sum_{\blam\in\sigma_N(\phi)} \tphi_{\blam} e^{i\blam\cdot \bx},
		\end{align*}
		where $\tphi_{\blam}$ is obtained by the discrete Fourier-Bohr transform of $\phi$,
		it follows that
		\begin{align*}
			I_N e^{i\frac{\tau}{2}\Delta} I_N \phi=\sum_{\blam\in\sigma_N(\phi)} \tphi_{\blam} e^{i\frac{\tau}{2}\Vert \blam\Vert} e^{i\blam\cdot \bx},
		\end{align*}
		which means that
		\begin{align*}
			\Vert I_N e^{i\frac{\tau}{2}\Delta} I_N \phi \Vert^2
			=\sum_{\blam\in\sigma_N(\phi)} \vert \tphi_{\blam}\vert^2
			=\Vert I_N \phi \Vert^2.
		\end{align*}
		
		(ii) Similar to the proof of Lemma \ref{lem:ABerror}, we consider the initial value problems \eqref{initialproblem-1} and \eqref{initialproblem-2} with solutions
		$\psi_1=e^{-itB(\phi)} \phi$ and $\psi_2=e^{-itB(\varphi)} \varphi.$
		Meanwhile, we consider the initial value problem
		\begin{align*}
			\begin{cases}
				i I_N\dfrac{d}{dt} (\psi_1-\psi_2)(t)= I_N(B(\phi)\psi_1(t)-B(\varphi)\psi_2(t)),\\
				I_N(\psi_1-\psi_2)(0)=I_N(\phi-\varphi).
			\end{cases}
		\end{align*}
		By integrating the above formula, we can obtain
		\begin{align*}
			i I_N(\psi_1-\psi_2)(t)=I_N(\phi-\varphi)+\int^t_0 I_N(B(\phi)\psi_1(\tau)-B(\varphi)\psi_2(\tau)) d\tau.
		\end{align*}
		Therefore, we denote $B_p(\varphi_p)=V_p+\theta\vert \varphi_p\vert^2$ to get
		\begin{align*}
			&\Vert  I_N(\psi_1-\psi_2)(t)\Vert\\
			&\leq \Vert I_N(\phi-\varphi)\Vert +\int^t_0 \Vert I_N(B(\phi)\psi_1(\tau)-B(\varphi)\psi_2(\tau))\Vert d\tau\\
			&\leq \Vert I_N(\phi-\varphi)\Vert +\int^t_0 \Vert I_NV(\psi_1-\psi_2) \Vert d\tau\\
			&~~+ \int^t_0 \vert \theta\vert (\Vert I_N \vert \phi\vert^2(\psi_1-\psi_2)\Vert
			+\Vert I_N \phi\psi_2 (\overline{\phi-\varphi}) \Vert +\Vert I_N(\phi-\varphi)\bar{\varphi}\psi_2(\tau)\Vert )d\tau\\
			&\leq \Vert I_N(\phi-\varphi)\Vert + \Vert V \Vert_{L_{QP}^{\infty}(\bbR^d)} \int^t_0 \Vert I_N(\psi_1-\psi_2) \Vert d\tau\\
			&~~+ \int^t_0 \vert \theta\vert \Big[\Vert \phi\Vert^2_{L_{QP}^{\infty}(\bbR^d)} \Vert I_N(\psi_1-\psi_2)\Vert+\Vert \phi\Vert_{L_{QP}^{\infty}(\bbR^d)}\Vert\psi_2\Vert_{L_{QP}^{\infty}(\bbR^d)} \Vert I_N(\phi-\varphi)\Vert\\
			&~~ 
			+\Vert \varphi\Vert_{L_{QP}^{\infty}(\bbR^d)} \Vert\psi_2\Vert_{L_{QP}^{\infty}(\bbR^d)} \Vert I_N(\phi-\varphi)\Vert  \Big] d\tau\\
			&=\Vert I_N(\phi-\varphi)\Vert + (\Vert V \Vert_{L_{QP}^{\infty}(\bbR^d)} +\vert \theta\vert \Vert \phi\Vert^2_{L_{QP}^{\infty}(\bbR^d)})\int^t_0 \Vert I_N(\psi_1-\psi_2) \Vert d\tau\\
			&~~+ (\Vert \phi\Vert_{L_{QP}^{\infty}(\bbR^d)}+\Vert \varphi\Vert_{L_{QP}^{\infty}(\bbR^d)}) \Vert I_N(\phi-\varphi)\Vert
			\int^t_0  \Vert e^{-i\tau B(\varphi)}\varphi\Vert_{L_{QP}^{\infty}(\bbR^d)}  d\tau \\
			&\leq \Vert I_N(\phi-\varphi)\Vert + (C_V +\vert \theta\vert C_p^2)
			\int^t_0 \Vert I_N(\psi_1-\psi_2) \Vert d\tau\\
			&~~+ 2 C\Vert I_N(\phi-\varphi)\Vert
			\int^t_0  \Vert e^{-i\tau B_p(\varphi_p)}\varphi_p\Vert_{X_{\alpha}}  d\tau 
			~~{\mbox{(Lemma \ref{lem:normineq}) }}\\
			&\leq  (C_V +\vert \theta\vert C_p^2)
			\int^t_0 \Vert I_N(\psi_1-\psi_2) \Vert d\tau
			+ \Big (1+2 C
			\int^t_0 \Vert \varphi_p\Vert_{X_{\alpha}}  d\tau \Big)\Vert I_N(\phi-\varphi)\Vert \\
			&\leq (C_V +\vert \theta\vert C_p^2)
			\int^t_0 \Vert I_N(\psi_1-\psi_2) \Vert d\tau
			+ (1+2 C_p^2t )\Vert I_N(\phi-\varphi)\Vert\\
			&\leq  (C_V +\vert \theta\vert C_p^2)
			\int^t_0 \Vert I_N(\psi_1-\psi_2) \Vert d\tau
			+ e^{ 2 C_p^2 t} \Vert I_N(\phi-\varphi)\Vert.
		\end{align*}
		Finally, we apply the Gronwall's inequality, see e.g. \cite[Lemma B.9]{Shen2011Spectral}, to prove (ii). 
	\end{proof}

	\begin{thm}
		\label{thm:phierror}
		Under the conditions of Lemma \ref{lem:trunAB}, the following estimate holds
		\begin{align*}
			\Vert \Upsilon_\phi \phi-\Upsilon_\varphi \varphi \Vert\leq e^{C(C_V+\vert \theta \vert C_p^2)\tau} \Vert I_N(\phi-\varphi) \Vert.
		\end{align*}
	\end{thm}

	\begin{proof}
		Applying Lemma \ref{lem:trunAB}, we have
		\begin{align*}
			\Vert \Upsilon_\phi \phi-\Upsilon_\varphi \varphi \Vert
			&= \Vert e^{\frac{i}{2}\tau \Delta}I_N [e^{-i\tau B(\phi_N^{**})} e^{\frac{i}{2}\tau \Delta} I_N\phi-e^{-i\tau B(\varphi_N^{**})} e^{\frac{i}{2}\tau \Delta} I_N\varphi] \Vert\\ 
			&=\Vert I_N (e^{-i\tau B(\phi_N^{**}) } e^{\frac{i}{2}\tau \Delta} I_N\phi-e^{-i\tau B(\varphi_N^{**})} e^{\frac{i}{2}\tau \Delta} I_N \varphi) \Vert\\
			&\leq e^{C(C_V+\vert \theta \vert C_p^2)\tau} \Vert I_N(\phi-\varphi) \Vert.
		\end{align*} 
		The proof of this theorem is completed.
	\end{proof}
	
	\begin{lemma}\cite[Lemma 5]{jiang2023High-accuracy}
		\label{lem:defectAIN-INA}
		For any $\phi\in \QP(\bbR^d)$,	it holds that
		\begin{align*}
			\Vert (I_N e^{\frac{i}{2}\tau \Delta}-e^{\frac{i}{2}\tau \Delta}I_N )\phi\Vert
			\leq \int^\tau_{0} \Vert[I_N-\calI,\Delta] e^{\frac{i}{2}\tau_1 \Delta} \phi \Vert \, d\tau_1,~~0<\tau_1 \leq \tau.
		\end{align*}
	\end{lemma}


	\begin{thm}
		\label{thm:space-error-part}
		Suppose $\phi\in \QP(\bbR^d)$ with  $\phi_p\in X_{\alpha}$ and $\Vert V_p \Vert_{X_\alpha} \leq C_V$ for some $\alpha>n/4$. Then the parent functions of $\phi^*= e^{\frac{i}{2}\tau \Delta}\phi$,
		$\phi_N^*=I_N e^{\frac{i}{2}\tau \Delta} \phi$
		and
		$\phi_N^{**}=e^{\frac{i}{2}\tau \Delta} I_N \phi$ belong to $X_{\alpha}$ and 
		\begin{align*}
			\Vert I_N\Gamma_\phi \phi-\Upsilon_\phi \phi \Vert \leq C\tau N^{-\alpha} 
			(\vert \phi_p\vert_\alpha+\vert \phi^*_p\vert_\alpha).
		\end{align*}
	\end{thm}

	\begin{proof}
		By Lemma \ref{lem:vp-norm} and the definition of the interpolation operator $I_N$,
		we have $\Vert \phi^*_p\Vert_{X_{\alpha}}=\Vert \phi_p \Vert_{X_{\alpha}}$ and
		\begin{align*}
			&\Vert \phi^*_{N,p}\Vert_{X_{\alpha}} =\Vert (e^{i\frac{\tau}{2}\Delta}I_N \phi)_p \Vert_{X_{\alpha}}
			= \Vert e^{i\frac{\tau}{2}\Delta}(I_N \phi)_p \Vert_{X_{\alpha}}
			=\Vert I_N\phi_p \Vert_{X_{\alpha}},\\
			&\Vert \phi^{**}_{N,p}\Vert_{X_{\alpha}} =\Vert (I_Ne^{i\frac{\tau}{2}\Delta}\phi)_p \Vert_{X_{\alpha}}
			= \Vert I_N(e^{i\frac{\tau}{2}\Delta}\phi)_p \Vert_{X_{\alpha}}.
		\end{align*}
		Note that the interpolation operator $I_N$ acts on periodic functions now such that we employ $\Vert \phi_p\Vert_{X_{\alpha}}\leq C$, Lemma \ref{lem:vp-norm} and the stability of the interpolation operator $I_N$ (see e.g. \cite{Shen2011Spectral} and \cite[Lemma 1]{Got}) to get $\Vert I_N\phi_p \Vert_{X_{\alpha}} \leq C$ and $\Vert I_N(e^{i\frac{\tau}{2}\Delta}\phi)_p \Vert_{X_{\alpha}} \leq C$, which proves the first statement of this theorem.
		
		We then split the difference as
		\begin{align*}
			I_N\Gamma_\phi \phi-\Upsilon_\phi \phi
			&=I_N e^{\frac{i}{2}\tau \Delta} e^{-i\tau B(\phi^*)} e^{\frac{i}{2}\tau \Delta}\phi
			-e^{\frac{i}{2}\tau \Delta}I_N e^{-i\tau B(\phi_N^{**})}I_N e^{\frac{i}{2}\tau \Delta} I_N \phi\\
			&=I_N e^{\frac{i}{2}\tau \Delta} e^{-i\tau B(\phi^*)} e^{\frac{i}{2}\tau \Delta}\phi-
			e^{\frac{i}{2}\tau \Delta} I_N e^{-i\tau B(\phi^*)} e^{\frac{i}{2}\tau \Delta} \phi\\
			&~~+e^{\frac{i}{2}\tau \Delta} I_N e^{-i\tau B(\phi^*)} e^{\frac{i}{2}\tau \Delta}\phi-
			e^{\frac{i}{2}\tau \Delta} I_N e^{-i\tau B(\phi_N^*) } I_N e^{\frac{i}{2}\tau \Delta} \phi\\
			&~~+e^{\frac{i}{2}\tau \Delta} I_N e^{-i\tau B(\phi_N^*) } I_N e^{\frac{i}{2}\tau \Delta}\phi
			-e^{\frac{i}{2}\tau \Delta}I_N e^{-i\tau B(\phi_N^{**})}I_N e^{\frac{i}{2}\tau \Delta} I_N \phi\\
			&=:Z_1+Z_2+Z_3.
		\end{align*}
		We apply Lemmas \ref{lem:trunAB} and \ref{lem:defectAIN-INA} to bound $Z_1$--$Z_3$ as
		\begin{align*}
			\Vert Z_1 \Vert 
			&= \Vert I_N e^{\frac{i}{2}\tau \Delta} e^{-i\tau B(\phi^*)} e^{\frac{i}{2}\tau \Delta}\phi-
			e^{\frac{i}{2}\tau \Delta} I_N e^{-i\tau B(\phi^*)} e^{\frac{i}{2}\tau \Delta}\phi \Vert\\
			&= \Vert (I_N e^{\frac{i}{2}\tau \Delta}-
			e^{\frac{i}{2}\tau \Delta}I_N)e^{-i\tau B(\phi^*)} e^{\frac{i}{2}\tau \Delta} \phi\Vert\\
			&\leq \int^\tau_{0} \Vert[I_N-\calI,\Delta] e^{\frac{i}{2}\tau_1 \Delta} e^{-i\tau_1 B(\phi^*)} e^{\frac{i}{2}\tau_1 \Delta}\phi \Vert \, d\tau_1\leq C \tau N^{-\alpha} \vert \phi_p \vert_{X_{\alpha}},\\
			\Vert Z_2\Vert 
			&=\Vert e^{\frac{i}{2}\tau \Delta} I_N e^{-i\tau B(\phi^*)} e^{\frac{i}{2}\tau \Delta}\phi-
			e^{\frac{i}{2}\tau \Delta} I_N e^{-i\tau B(\phi_N^*)} I_N e^{\frac{i}{2}\tau \Delta}\phi\Vert\\
			&=\Vert I_N e^{-i\tau B(\phi^*)} e^{\frac{i}{2}\tau \Delta}\phi-
			I_N e^{-i\tau B(\phi_N^*)} I_N e^{\frac{i}{2}\tau \Delta}\phi\Vert\\
			&\leq e^{C(C_V+\vert \theta\vert C_p^2)}
			\Vert I_N e^{\frac{i}{2}\tau \Delta}\phi-I_N e^{\frac{i}{2}\tau \Delta}\phi\Vert
			=0,\\
			\Vert Z_3\Vert
			&=\Vert e^{\frac{i}{2}\tau \Delta} I_N e^{-i\tau B(\phi_N^*) } I_N e^{\frac{i}{2}\tau \Delta}\phi
			-e^{\frac{i}{2}\tau \Delta}I_N e^{-i\tau B(\phi_N^{**})}I_N e^{\frac{i}{2}\tau \Delta} I_N \phi\Vert\\
			&=\Vert I_N e^{-i\tau B(\phi_N^*)} I_N e^{\frac{i}{2}\tau \Delta}\phi
			-I_N e^{-i\tau B(\phi_N^{**})}I_N e^{\frac{i}{2}\tau \Delta} I_N \phi\Vert\\
			&\leq e^{C(C_V+\vert \theta \vert C_p^2)}\Vert I_N e^{\frac{i}{2}\tau \Delta}\phi
			- e^{\frac{i}{2}\tau \Delta} I_N\phi\Vert\leq C\tau N^{-\alpha} \vert \phi^*_p\vert_{X_{\alpha}},
		\end{align*}
		which completes the proof.
	\end{proof}

	\subsection{Estimate of intermediate solutions}
	
	We analyze the controllability of some intermediate variables.
	
	\begin{lemma}
		\label{lem:normeq}
		For $\ell=1,\cdots,m$, define 
		\begin{align*}
			\psi_{\ell}=\sum_{\bk\in\bbZ^n} \hpsi^{(\ell)}_{\bk} e^{i\theta_{\ell}}e^{i\bk\cdot \by},~~
			\varphi_{\ell}=\sum_{\bk\in\bbZ^n} \hpsi^{(\ell)}_{\bk} e^{i\tilde\theta_{\ell}}e^{i\bk\cdot \by},~~\theta_{\ell},~ \tilde\theta_{\ell}\in \bbR.
		\end{align*}
		Then, the following equality holds
		\begin{align*}
			\Big \Vert \sum_{\ell=1}^{m} \sum_{j_{\ell}=1}^{m}\cdots\sum_{j_{1}=1}^{m}\psi_{j_1}\psi_{j_2}\cdots\psi_{j_\ell} \Big \Vert
			=\Big \Vert \sum_{\ell=1}^{m} \sum_{j_{\ell}=1}^{m}\cdots\sum_{j_{1}=1}^{m}
			\varphi_{j_1} \varphi_{j_2}\cdots \varphi_{j_\ell} \Big \Vert.
		\end{align*}
	\end{lemma}
	
	\begin{proof}
		According to the definition, we have
		\begin{align*}
			&\sum_{\ell=1}^{m} \sum_{j_{\ell}=1}^{m}\cdots\sum_{j_{1}=1}^{m}
			\psi_{j_1}\psi_{j_2}\cdots \psi_{j_\ell}\\
			&=\sum_{\ell=1}^{m} \sum_{j_{\ell}=1}^{m}\cdots\sum_{j_{1}=1}^{m}
			\Big (\sum_{\bk_1\in\bbZ^n} \hpsi^{(j_1)}_{\bk_1} e^{i\theta_{j_1}}e^{i\bk_1\cdot \by} \Big)
			\Big(\sum_{\bk_2\in\bbZ^n} \hpsi^{(j_2)}_{\bk_2} e^{i\theta_2}e^{i\bk_2\cdot \by} \Big )
			\cdots
			\Big(\sum_{\bk_{j_\ell}\in\bbZ^n} \hpsi^{(j_\ell)}_{\bk_{j_\ell}} e^{i\theta_{j_\ell}}e^{i\bk_{j_\ell}\cdot \by} \Big )\\
			&=\sum_{\ell=1}^{m} \sum_{j_{\ell}=1}^{m}\cdots\sum_{j_{1}=1}^{m}
			\sum_{\bk_1,\ldots,\bk_{j_\ell}} 
			\hpsi^{(j_1)}_{\bk_1} \hpsi^{(j_2)}_{\bk_2}\cdots 
			\hpsi^{(j_\ell)}_{\bk_{j_\ell}}
			e^{i(\theta_{j_1}+\theta_{j_2}+\cdots+\theta_{j_\ell})}e^{i(\bk_{j_1}+\bk_{j_2}+\cdots+\bk_{j_\ell})\cdot \by} 
		\end{align*}
		and similarly,
		\begin{align*}
			&\sum_{\ell=1}^{m}
			\sum_{j_{\ell}=1}^{m}\cdots\sum_{j_{1}=1}^{m}
			\varphi_{j_1} \varphi_{j_2}\cdots \varphi_{j_\ell}\\
			&\qquad=\sum_{\ell=1}^{m}
			\sum_{j_{\ell}=1}^{m}\cdots\sum_{j_{1}=1}^{m}
			\sum_{\bk_1,\ldots,\bk_{j_\ell}} 
			\hpsi^{(j_1)}_{\bk_{j_1}}\hpsi^{(j_2)}_{\bk_{j_2}}\cdots 
			\hpsi^{(j_\ell)}_{\bk_{j_\ell}}
			e^{i(\tilde\theta_{j_1}+\tilde\theta_{j_2}+\cdots+\tilde\theta_{j_\ell})}e^{i(\bk_{j_1}+\bk_{j_2}+\cdots+\bk_{j_\ell})\cdot \by}. 
		\end{align*}
		Since $\vert e^{i \theta}\vert =1$ for any $\theta\in\mathbb R$, 
		we reach the conclusion.
	\end{proof}
	
	
	\begin{lemma}
		\label{lem:eq-inter-values}
		Assume that $\sup\{\Vert \psi_p(t)\Vert_{X_\alpha}:~0\leq t\leq T\}\leq C_p$ for some $\alpha>\max\{4,n/4\}$, then the intermediate values with $1\leq m\leq M$ could be bounded as
		\begin{align*}
			\tilde{a}(m-1)&=\sup_{\ell=1,2,\cdots,m-1,\atop j=\ell,\cdots,m-1}\{\Vert \psi_p(t_{\ell}) \Vert_{X_{\alpha}}, 
			\Vert (\Pi^{m}_{j=\ell}\Gamma_{\ell-1,j-1}\psi(t_{\ell-1}))_p \Vert_{X_{\alpha}}
			\}\leq C_p,\\
			\tilde{b}(m-1)&=\sup_{\ell=1,2,\cdots,m-1,\atop j=\ell,\cdots,m-1}\{\Vert (I_N \psi_{\ell})_p \Vert_{X_{\alpha}}, 
			\Vert (\Pi^{m}_{j=\ell}\Upsilon_{\ell-1,j-1}\psi_{\ell-1})_p \Vert_{X_{\alpha}}
			\}\leq C_p.
		\end{align*}
		
	\end{lemma}
	
	\begin{proof}
		Define a high-dimensional auxiliary periodic nonlinear Schr\"odinger equation
		\begin{align}
			\begin{cases}
				i\dfrac{\partial W}{\partial t}=-\Delta W+V_p W+\theta \vert W\vert^2 W ,~~(\by,t)\in\bbT^n\times[T_0,T],\\
				W(T_0)=W(\by, T_0).
			\end{cases}
			\label{eq:QSE-parent}
		\end{align}
		We apply the Strang splitting method in time and the Fourier pseudo-spectral method in space to solve \eqref{eq:QSE-parent}, which leads to the semidiscrete-in-time  and the fully-discrete numerical schemes
		\begin{align*}
			&W_m=\Pi_{j=1}^{m} e^{i\frac{\tau}{2}\Delta}e^{-i\tau (V_p+\theta\vert W^*_{j-1}\vert^2)}e^{i\frac{\tau}{2}\Delta}W_{0},\\
			&W_{N,m}=\Pi_{j=1}^{m} e^{i\frac{\tau}{2}\Delta}I_Ne^{-i\tau (V_p+\theta\vert W^*_{N,j-1}\vert^2)}e^{i\frac{\tau}{2}\Delta}I_NW_{0},
		\end{align*}
		where $W^*_{j-1}=e^{i\frac{\tau}{2}\Delta}W_{j-1}$, $W^*_{N,j-1}=e^{i\frac{\tau}{2}\Delta}I_NW_{j-1}$ and $W_0=W(T_0)$.
		When we take $W_0=\psi_p(t_{\ell-1})$ with $T_0=t_\ell$ for $\ell=1,\cdots,m$, respectively, different numerical solutions can be obtained and the sets of the upper bounds of these numerical solutions are
		\begin{align*}
			&\tilde{a}_p(m-1)=\sup_{\ell=1,2,\cdots,m-1,\atop j=\ell,\cdots,m-1}\{\Vert \psi_p(t_{\ell}) \Vert_{X_{\alpha}}, 
			\Vert \Pi^{m}_{j=\ell}\Gamma_{\ell-1,j-1,p}\psi_p(t_{\ell-1}) \Vert_{X_{\alpha}}
			\},\\
			&\Gamma_{\ell-1,j-1,p} \psi_p(t_{\ell-1})= e^{i\frac{\tau}{2}\Delta}e^{-i\tau (V_p+\theta\vert W^*_{j-1}\vert^2)}e^{i\frac{\tau}{2}\Delta}\psi_p(t_{\ell-1}),
		\end{align*}
		and 
		\begin{align*}
			&\tilde{b}_p(m-1)
			=\sup_{\ell=1,2,\cdots,m-1,\atop j=\ell,\cdots,m-1}\{\Vert I_N \psi_{\ell,p} \Vert_{X_{\alpha}},
			\Vert \Pi^{m}_{j=\ell}\Upsilon_{\ell-1,j-1,p}W_{\ell-1}\Vert_{X_{\alpha}}\},\\
			&\Psi^*_{\ell-1}=e^{i\frac{\tau}{2}\Delta}\psi_p(t_{\ell-1}),~~W_{\ell-1}=e^{i\frac{\tau}{2}\Delta}e^{-i\tau (V_p+\theta\vert \Psi^*_{\ell-1}\vert^2)}e^{i\frac{\tau}{2}\Delta}\psi_p(t_{\ell-1}) ,\\
			&\Upsilon_{\ell-1,j-1,p} W_{\ell-1}
			= e^{i\frac{\tau}{2}\Delta} I_Ne^{-i\tau (V_p+\theta\vert W^*_{N,\ell-1}\vert^2)}e^{i\frac{\tau}{2}\Delta} I_N W_{\ell-1}.
		\end{align*}
		By \cite[Theorems 3.1 and 3.4]{Gauckler2010Convergence}, $\tilde{a}_p(m-1)\leq C_p$ and $\tilde{b}_p(m-1)\leq C_p$ is proved when $\sup\{\Vert \psi_p(t)\Vert_{X_\alpha}:~0\leq t\leq T\}\leq C_p$ where $C_p$ depends on $d$, $\alpha$ and $T$.
		Thus we remain to show
		\begin{align}
			\tilde{a}(m-1)=\tilde{a}_p(m-1)
			~~\mbox{and}~~\tilde{b}(m-1)=\tilde{b}_p(m-1),
			\label{eq:Object}
		\end{align}
		such that the proof could be completed.
		We focus on the first relation  $\tilde{a}(m-1)=\tilde{a}_p(m-1)$ in (\ref{eq:Object}) since the second relation, which could be viewed as a truncated version of the first one, could be proved similarly. To prove the first relation, it suffices to show
		\begin{align}
			\Vert ( e^{i\frac{\tau}{2}\Delta}e^{-i\tau (V+\theta\vert \psi^*\vert^2)}e^{i\frac{\tau}{2}\Delta}\psi)_p \Vert_{X_{\alpha}} 
			=\Vert  e^{i\frac{\tau}{2}\Delta}e^{-i\tau (V_p+\theta\vert \psi_p^*\vert^2)}e^{i\frac{\tau}{2}\Delta}\psi_p \Vert_{X_{\alpha}},
			\label{eq:Xalpha}
		\end{align}
		which is equivalent to
		\begin{align*}
			\Vert ( e^{i\frac{\tau}{2}\Delta}e^{-i\tau (V+\theta\vert \psi^*\vert^2)}e^{i\frac{\tau}{2}\Delta}\psi)_p \Vert 
			=\Vert  e^{i\frac{\tau}{2}\Delta}e^{-i\tau (V_p+\theta\vert \psi_p^*\vert^2)}e^{i\frac{\tau}{2}\Delta}\psi_p \Vert,\\
			\Vert \Delta^{\alpha}( e^{i\frac{\tau}{2}\Delta}e^{-i\tau (V+\theta\vert \psi^*\vert^2)}e^{i\frac{\tau}{2}\Delta}\psi)_p \Vert 
			=\Vert  \Delta^{\alpha} e^{i\frac{\tau}{2}\Delta}e^{-i\tau (V_p+\theta\vert \psi_p^*\vert^2)}e^{i\frac{\tau}{2}\Delta}\psi_p \Vert,
		\end{align*}
		where $\psi^*=e^{i\frac{\tau}{2}\Delta}\psi$ and $\psi^*_p=e^{i\frac{\tau}{2}\Delta}\psi_p$.
		By Lemma \ref{lem:vp-norm} we have
		\begin{align*}
			&\Vert ( e^{i\frac{\tau}{2}\Delta}e^{-i\tau (V+\theta\vert \psi^*\vert^2)}e^{i\frac{\tau}{2}\Delta}\psi)_p \Vert \\	
			&=\Vert (e^{-i\tau (V+\theta\vert \psi^*\vert^2)}e^{i\frac{\tau}{2}\Delta}\psi)_p \Vert=\Vert e^{-i\tau (V_p+\theta\vert \tilde \psi_p^*\vert^2)}(e^{i\frac{\tau}{2}\Delta}\psi)_p \Vert\\
			&=\Vert(e^{i\frac{\tau}{2}\Delta}\psi)_p \Vert=\Vert \psi_p \Vert
			=\Vert  e^{i\frac{\tau}{2}\Delta}e^{-i\tau (V_p+\theta\vert \psi_p^*\vert^2)}e^{i\frac{\tau}{2}\Delta}\psi_p \Vert,
		\end{align*}
		where $\tilde{\psi}_p^*=(e^{i\frac{\tau}{2}\Delta}\psi)_p$.
		Secondly,
		\begin{align*}
			\Vert \Delta^{\alpha}( e^{i\frac{\tau}{2}\Delta}e^{-i\tau (V+\theta\vert \psi^*\vert^2)}e^{i\frac{\tau}{2}\Delta}\psi)_p \Vert 
			&=\Vert \Delta^{\alpha}e^{i\frac{\tau}{2}\Delta}(e^{-i\tau (V+\theta\vert \psi^*\vert^2)}e^{i\frac{\tau}{2}\Delta}\psi)_p \Vert \\
			&=\Vert e^{i\frac{\tau}{2}\Delta}\Delta^{\alpha}(e^{-i\tau (V+\theta\vert \psi^*\vert^2)}e^{i\frac{\tau}{2}\Delta}\psi)_p \Vert \\
			&= \Vert \Delta^{\alpha} (e^{-i\tau (V+\theta\vert \psi^*\vert^2)}e^{i\frac{\tau}{2}\Delta}\psi)_p \Vert\\
			&= \Vert \Delta^{\alpha} e^{-i\tau (V_p+\theta\vert \tilde{\psi}_p^*\vert^2)}\tilde{\psi}^*_p \Vert,
		\end{align*}
		and
		\begin{align*}
			\Vert  \Delta^{\alpha} e^{i\frac{\tau}{2}\Delta}e^{-i\tau (V_p+\theta\vert \psi_p^*\vert^2)}e^{i\frac{\tau}{2}\Delta}\psi_p \Vert 
			= \Vert \Delta^{\alpha} e^{-i\tau (V_p+\theta\vert \psi_p^*\vert^2)}\psi^*_p \Vert.
		\end{align*}
		Meanwhile, for 
		$\psi=\sum_{\blam\in \sigma(\psi)} \hat{\psi}_{\blam}e^{i\blam \cdot \bx},$
		then
		$\psi_p=\sum_{\bk\in \bK} \hat{\psi}_{\bk}e^{i\bk \cdot \by},~~\hat{\psi}_{\bk}=\hat{\psi}_{\blam},$
		and
		\begin{align*}
			\tilde{\psi}_p^*=\sum_{\bk\in \bK} \hat{\psi}_{\bk} e^{i\frac{\tau}{2}\Vert \blam\Vert^2} e^{i\bk \cdot \by},~~
			\psi_p^*=\sum_{\bk\in \bK} \hat{\psi}_{\bk} e^{i\frac{\tau}{2}\Vert \bk\Vert^2} e^{i\bk \cdot \by}.
		\end{align*}
		When $\alpha=1$, we have
		\begin{align*}
			\Delta (e^{-i\tau (V_p+\theta\vert \tilde{\psi}_p^*\vert^2)}\tilde{\psi}^*_p) 
			=e^{-i\tau (V_p+\theta\vert \tilde{\psi}_p^*\vert^2)}\{ [V'_p+\theta(\vert \tilde{\psi}_p^*\vert^2)']\tilde{\psi}^*_p+(\tilde{\psi}^*_p)'\}.
		\end{align*}
		When $\alpha>1$, the expansion expression is similar, which we will not elaborate on here.
		Applying Lemma \ref{lem:normeq}, it follows that
		\begin{align*}
			\Vert \Delta^{\alpha} e^{-i\tau (V_p+\theta\vert \tilde{\psi}_p^*\vert^2)}\tilde{\psi}^*_p \Vert
			= \Vert \Delta^{\alpha} e^{-i\tau (V_p+\theta\vert \psi_p^*\vert^2)}\psi^*_p \Vert,
		\end{align*}
		therefore the equation \eqref{eq:Xalpha} holds. This means that the equation \eqref{eq:Object} is true and this lemma is proved.
	\end{proof}
	


	\section{Numerical implementation}
	\label{sec:Num}
	
	We perform numerical experiments to substantiate the effectiveness and accuracy of the fully discrete scheme (\ref{eq:splitting-PM}).
	All algorithms are implemented using MSVC++ 14.29 on Visual Studio Community 2019. The FFT is implemented by the software FFTW 3.3.5\,\cite{frigo2005design}. All computations are performed on a workstation with an Intel Core 2.30GHz CPU, 16GB RAM. Denote the computational time in seconds by CPU(s), and the $L_{QP}^2$-norm error of the numerical solution  is computed by
	$\Err^\tau_{h}=\Vert \Psi_M- \psi(\cdot,t_M)\Vert$.
	The temporal convergence order is calculated by 
	$\kappa=\ln(\Err^{\tau_1}_{h}/\Err^{\tau_2}_{h})/\ln(\tau_{1}/\tau_{2}).
	$
	
	\subsection{One-dimensional case}
	Let $T=0.001$ with the incommensurate potential and initial value 
	\begin{align*}
		V(x)=\sin x + \sum_{k=1}^{4}\sin (k\sqrt{3}x),~~
		\psi_0(x)=\sum_{\lambda\in \sigma(\psi_0)}\hat{\psi}_{\lambda}e^{i\lambda x},
	\end{align*}
	where $\hat{\psi}_{\lambda}=e^{-(\vert m_1\vert +\vert m_2\vert)}$ and $\sigma(\psi_0)=\{\lambda = m_1+m_2\sqrt{3}: m_1,m_2\in \bbZ, -32\leq m_1,m_2\leq 31\}$. 
	Then
	$
	V_p(y_1,y_2)=\sin y_1 + \sum_{k=1}^{4}\sin (k y_2)$ for $ (y_1, y_2)\in\bbR^2/2\pi\bbZ^2.
	$
	The exact solution in $\Err^\tau_{h}$ is replaced by the numerical solution with a very small time step size $\tau=1\times 10^{-6}$ and $h=\pi/64$. 
	Numerical results are presented in Tables \ref{tab:error}-\ref{tab:timeerror}, which demonstrate the efficiency of the proposed method and the exponential convergence in space and the second-order accuracy in time, as proved in Theorem \ref{thm:PMerror}.
	
	\begin{table}[h]
		\centering
		\footnotesize{
			\caption{Error $\Err^\tau_{h}$ and CPU time for $\tau=1\times 10^{-6}$ and different $N$ and $\theta$.}
			\vspace{0.2cm}
			\label{tab:error}
			\renewcommand\arraystretch{1.8}
			\setlength{\tabcolsep}{3mm}
			{\begin{tabular}{|c|c|c|c|c|c|c|}\hline
					& $2N \times 2N$ &$4\times 4$  & $8\times 8$   & $16\times 16$ &$32\times 32$  &$64\times 64$    \\ \hline
					\multirow{2}*{$\theta$=0} & $\Err^\tau_{h}$  &  2.784e-03 & 5.080e-04  &  1.692e-05  &  1.134e-08  & 6.923e-14  \\ \cline{2-7}
					& CPU(s)  &0.005      &    0.01  &  0.027     & 0.100    &0.360    \\ \hline
					\multirow{2}*{$\theta$=1} & $\Err^\tau_{h}$  &  2.783e-03 & 5.087e-04  &  1.710e-05  &  1.254e-08  & 5.405e-14  \\ \cline{2-7}
					& CPU(s)  &0.006      &    0.009  &  0.025     & 0.098    &0.366     \\ \hline
					\multirow{2}*{$\theta$=10}&$\Err^\tau_{h}$  &  2.810e-03 & 5.689e-04  &  2.832e-05  &  5.213e-08  & 7.805e-14  \\ \cline{2-7}
					& CPU(s)  &0.006      &    0.009  &  0.027     & 0.094    &0.393    \\ \hline
					\multirow{2}*{$\theta$=20}&$\Err^\tau_{h}$ &  2.935e-03 & 7.364e-04  &  5.075e-05  &  1.237e-07  & 2.715e-13  \\ \cline{2-7}
					& CPU(s)  &0.005      &    0.011  &  0.024     & 0.102    &0.366    \\ \hline
			\end{tabular}}
		}
	\end{table}
	
	\begin{table}[h] 
		\centering
		\footnotesize{
			\caption{Error $\Err^\tau_{h}$ for $h=\pi/64$ and $\theta=10$. }\label{tab:timeerror}
			\vspace{0.1cm}
			\renewcommand\arraystretch{1.8}
			\setlength{\tabcolsep}{3mm}
			{\begin{tabular}{|c|c|c|c|c|c|}\hline
					$\tau$   &$1\times 10^{-3}$   &$5\times 10^{-4}$  &  $1\times 10^{-4}$  & $5\times 10^{-4}$  &$1\times 10^{-5}$\\ \hline
					$\Err^\tau_{h}$& 4.050e-05        & 1.615e-06 & 4.378e-07   & 1.611e-08   &3.798e-09     \\ \hline
					$\kappa$  &- &      2.00 &2.00   & 2.00 &2.00    \\ \hline
			\end{tabular}}
		}
	\end{table}
	
	Physically, solitons are localized nonlinear modes, which is found in quasiperiodic Schr\"odinger systems \cite{Ablowitz2006Solitons,Berti2012Sobolev}.
	To observe the soliton, we take the Gaussian-type initial value $\psi_p(\by, 0) = e^{-(y^2_1+y^2_2)/2}$ for $ \by\in [0,2\pi)^2$, $\theta=1$, $h=\pi/64$ and $\tau=10^{-6}$. Figure \ref{fig:soliton-EG1} shows the evolution of  the soliton at the origin.

	\begin{figure}[h]
		\centering
		\subfigure[$T=0.01$]{
			\includegraphics[width=1.62in,height=1.6in]{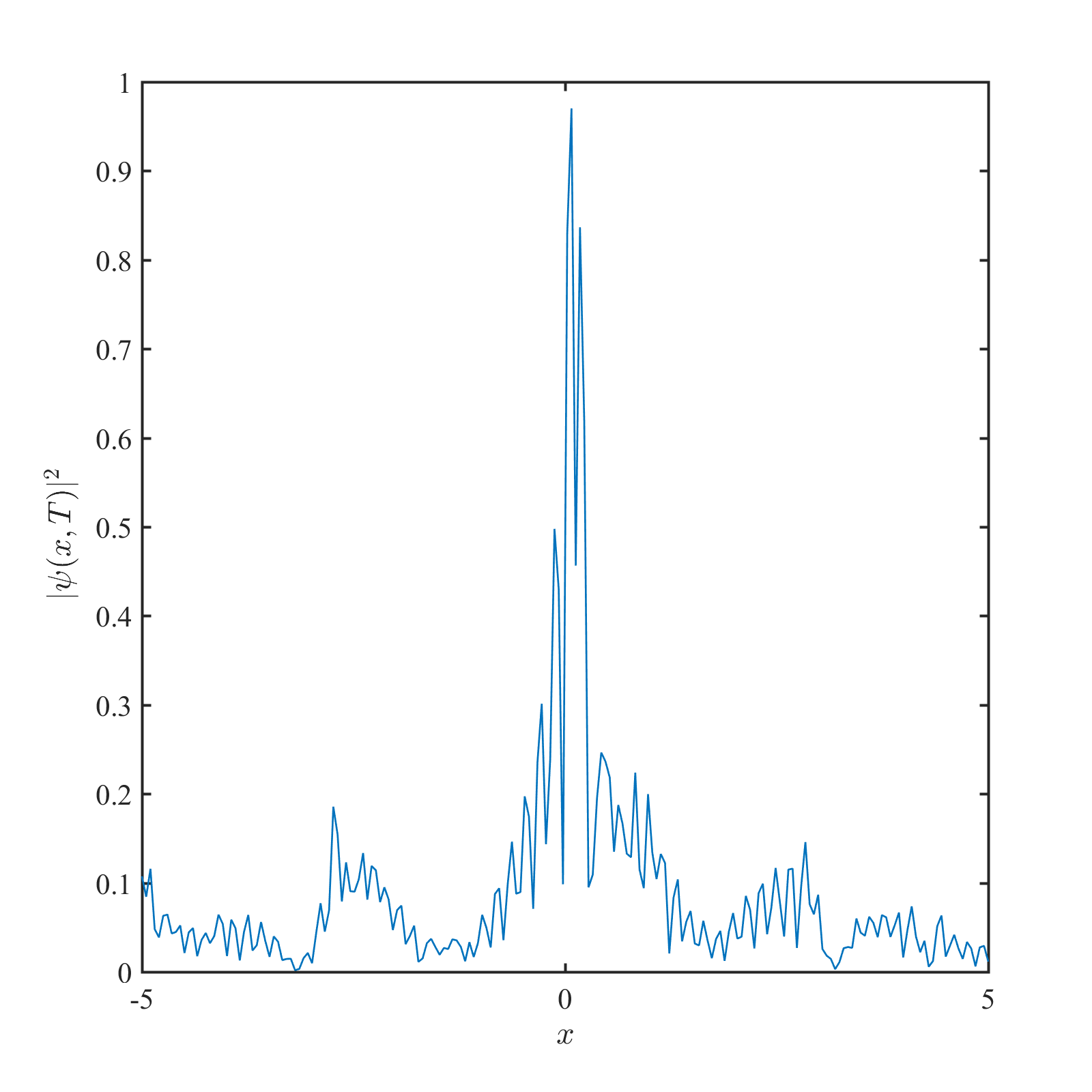}}
		\subfigure[$T=0.005$]{
			\includegraphics[width=1.62in,height=1.6in]{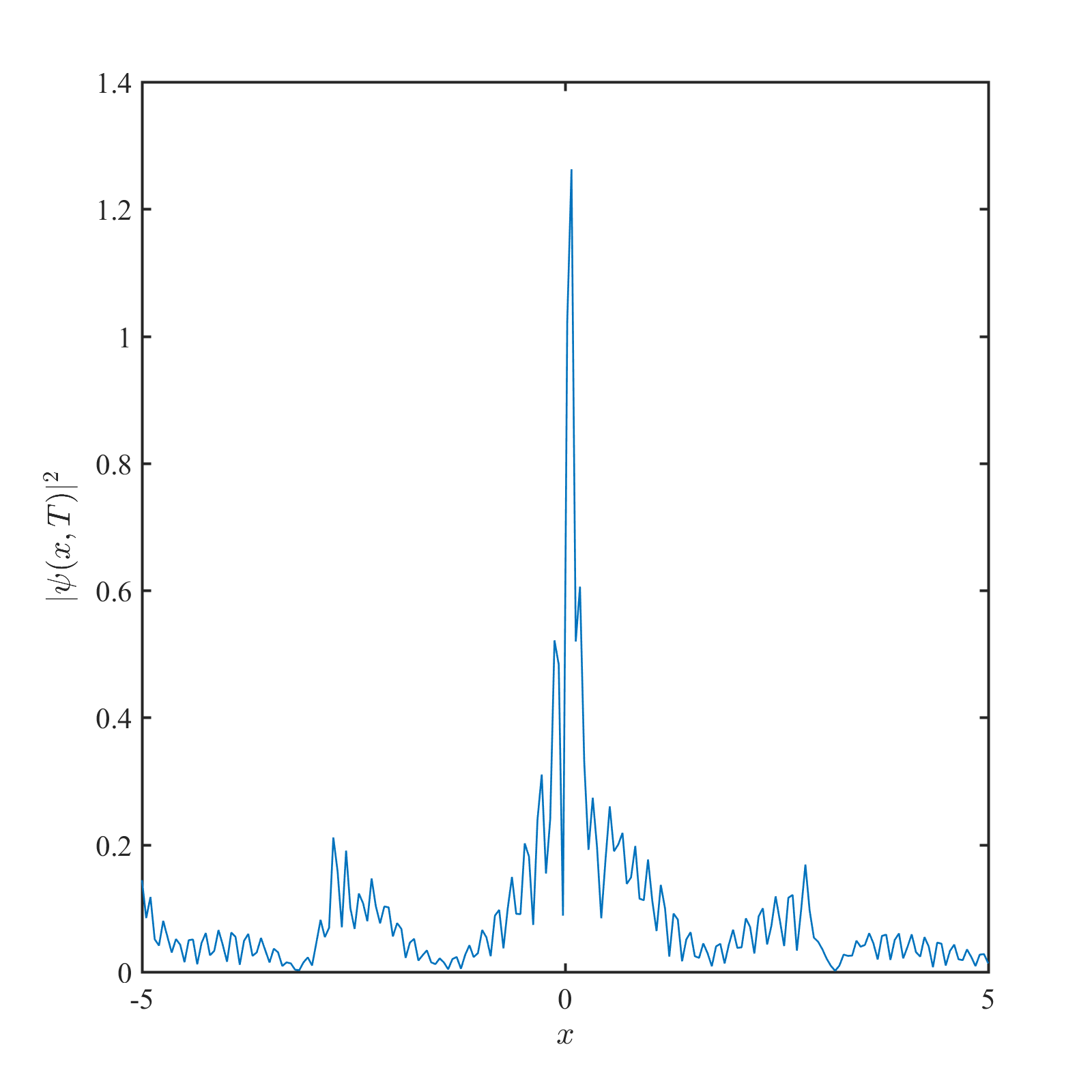}}
		\subfigure[$T=0.001$]{
			\includegraphics[width=1.62in,height=1.6in]{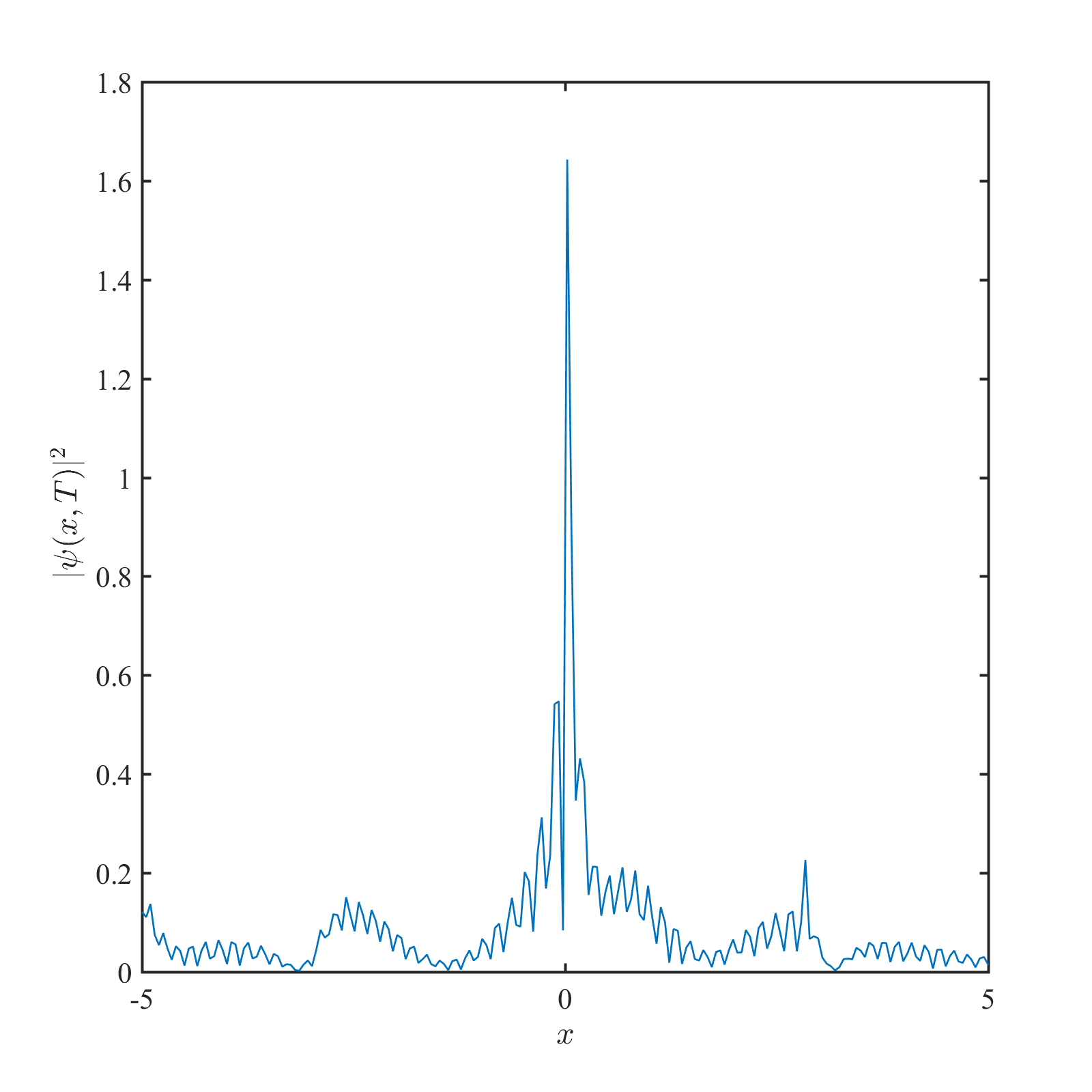}
		}
		\caption{The soliton at different time $T$.} \label{fig:soliton-EG1}
	\end{figure}

	
	\subsection{Two-dimensional case}
	Let $T=0.0001$ and consider the real-value potential function
	\begin{align*}
		V(\bx)=\sum_{j=1}^{3} \cos{(\bP\bk_j)\cdot \bx},	
	\end{align*}
	where the projection matrix
	\begin{align*}
		\bP = \begin{pmatrix}
			1 &\ds \cos\frac{\pi}{6} &\ds \sin\frac{\pi}{6} & 0\\[0.1in]
			0 &\ds \sin\frac{\pi}{6} &\ds \cos\frac{\pi}{6} & 1
		\end{pmatrix},~~(\bk_1, \bk_2, \bk_3)= 
		\begin{pmatrix}
			0 & 0 & 2 \\ 
			1 & -1& 0 \\
			0 & 3 & 0 \\
			-1 & 0 & 1 
		\end{pmatrix}.
	\end{align*} 
	Then we have
	$\ds V_p(\by)=\sum_{j=1}^{3} \cos{ (\bk_j\cdot \by)}$ for $\by\in \mathbb{T}^4 = \bbR^4/2\pi\bbZ^4.$
	Let the initial value 
	\begin{align*}
		\psi_0(\bx)=\sum_{\blam\in \sigma(\psi_0)}\hat{\psi}_{\blam}e^{i\blam\cdot \bx},
	\end{align*}
	where $\hat{\psi}_{\bm \lambda}=e^{-(\vert k_1\vert +\vert k_2\vert+\vert k_3\vert+\vert k_4\vert )}$, 
	and $\sigma(\psi_0)=\{\blam = \bP\bk: k_1, k_2, k_3\in \bbZ, -16\leq k_1, k_2, k_3\leq 15\}$.
	Therefore, this quasiperiodic system corresponds to a four-dimensional parent system.
	
	The exact solution in $\Err^\tau_{h}$ is replaced by the numerical solution under a very small time step size $\tau=1\times 10^{-7}$ and $h=\pi/32$. Numerical results are presented in Tables \ref{tab:error-2d}-\ref{tab:timeerror-2d}, which again verify the error estimates in Theorem \ref{thm:PMerror}.
	We also observe that the solitons appear in the center of the origin shown in Figure \ref{fig:soliton-EG2}, where the the quasiperiodic potential attains the global maximum. Meanwhile, these solitons have a dimple, which is similar to that found in two-dimensional Penrose lattices in \cite{Ablowitz2006Solitons}.

	\begin{table}[h]
		\centering
		\footnotesize{
			\caption{Error $\Err^\tau_{h}$ and CPU time for $\tau=1\times 10^{-7}$, $\theta=1$ and different $N$.}
			\vspace{0.1cm}
			\label{tab:error-2d}
			\renewcommand\arraystretch{1.8}
			\setlength{\tabcolsep}{2.8mm}
			{\begin{tabular}{|c|c|c|c|c|}\hline
					$2N \times 2N \times 2N \times 2N$ &$4\times 4\times 4\times 4$  & $8\times 8\times 8\times 8$   & $16\times 16\times 16\times 16$ &$32\times 32\times 32\times 32$     \\ \hline
					$\Err^\tau_{h}$  &  3.948e-04 & 5.221e-05  & 2.625e-07  & 9.941e-11  \\ \cline{1-5}
					CPU(s)  &0.029     &  0.376  &  7.612    & 119.641      \\ \hline
			\end{tabular}}
		}
	\end{table}

	\begin{table}[h] 
		\centering
		\footnotesize{
			\caption{Error $\Err^\tau_{h}$  for $h=\pi/8$ and $\theta=1$. }\label{tab:timeerror-2d}
			\vspace{0.1cm}
			\renewcommand\arraystretch{1.8}
			\setlength{\tabcolsep}{3mm}
			{\begin{tabular}{|c|c|c|c|c|c|}\hline
					$\tau$   &$1\times 10^{-4}$   &$2\times 10^{-5}$  &  $1\times 10^{-5}$  & $2\times 10^{-6}$  &$1\times 10^{-6}$\\ \hline
					$\Err^\tau_{h}$& 8.253e-08        & 3.301e-09 & 8.252e-10   & 3.293e-11   &8.174e-12     \\ \hline
					$\kappa$  &- &      2.00 &2.00   & 2.00 &2.01   \\ \hline
			\end{tabular}}
		}
	\end{table}

	\begin{figure}[h]
		\centering
		\subfigure[Cross section along the $x_2$ axis of a soliton]{
			\includegraphics[width=2.2in,height=2in]{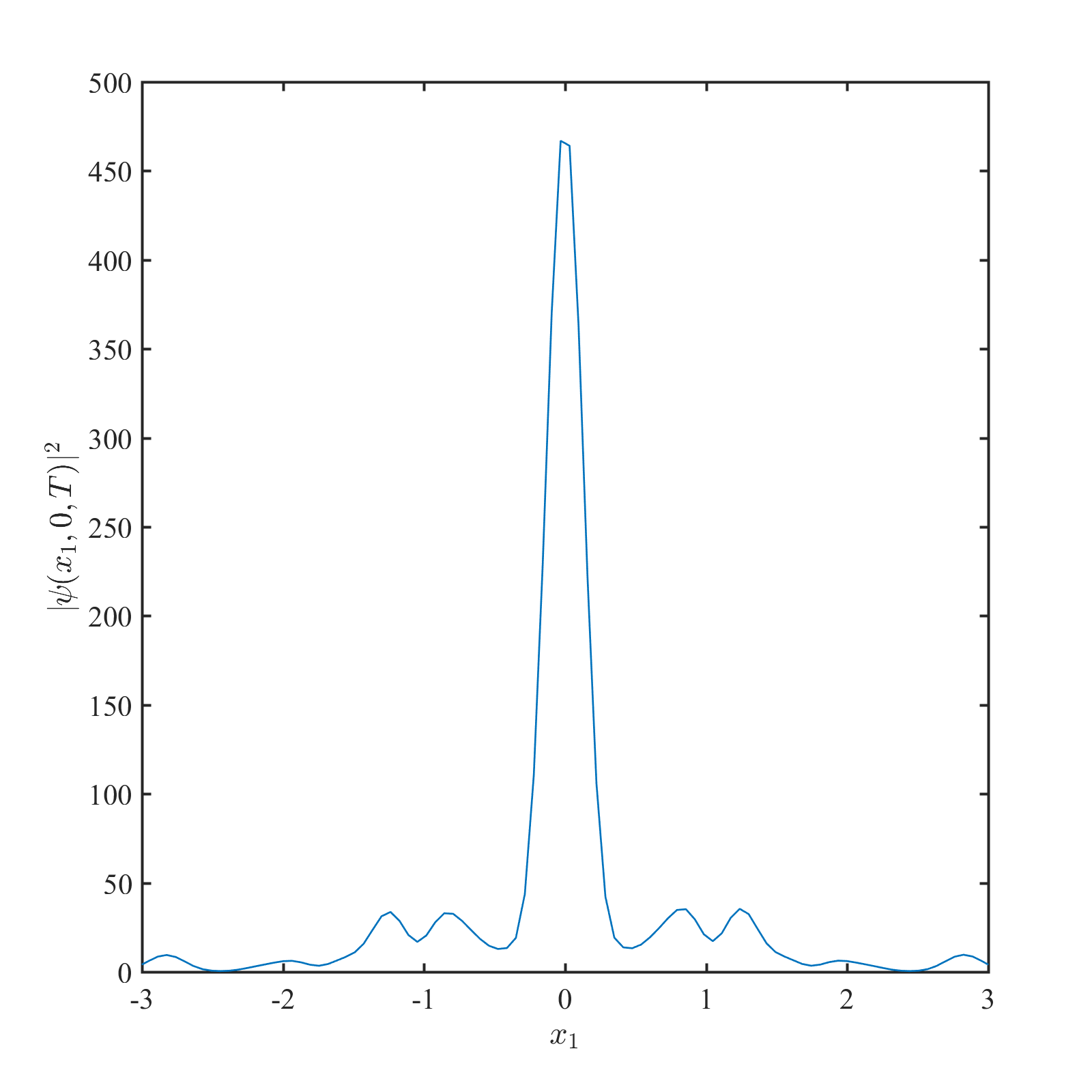}}
		\hspace{0.3in}
		\subfigure[3D view of a soliton’s intensity]{
			\includegraphics[width=2.2in,height=2in]{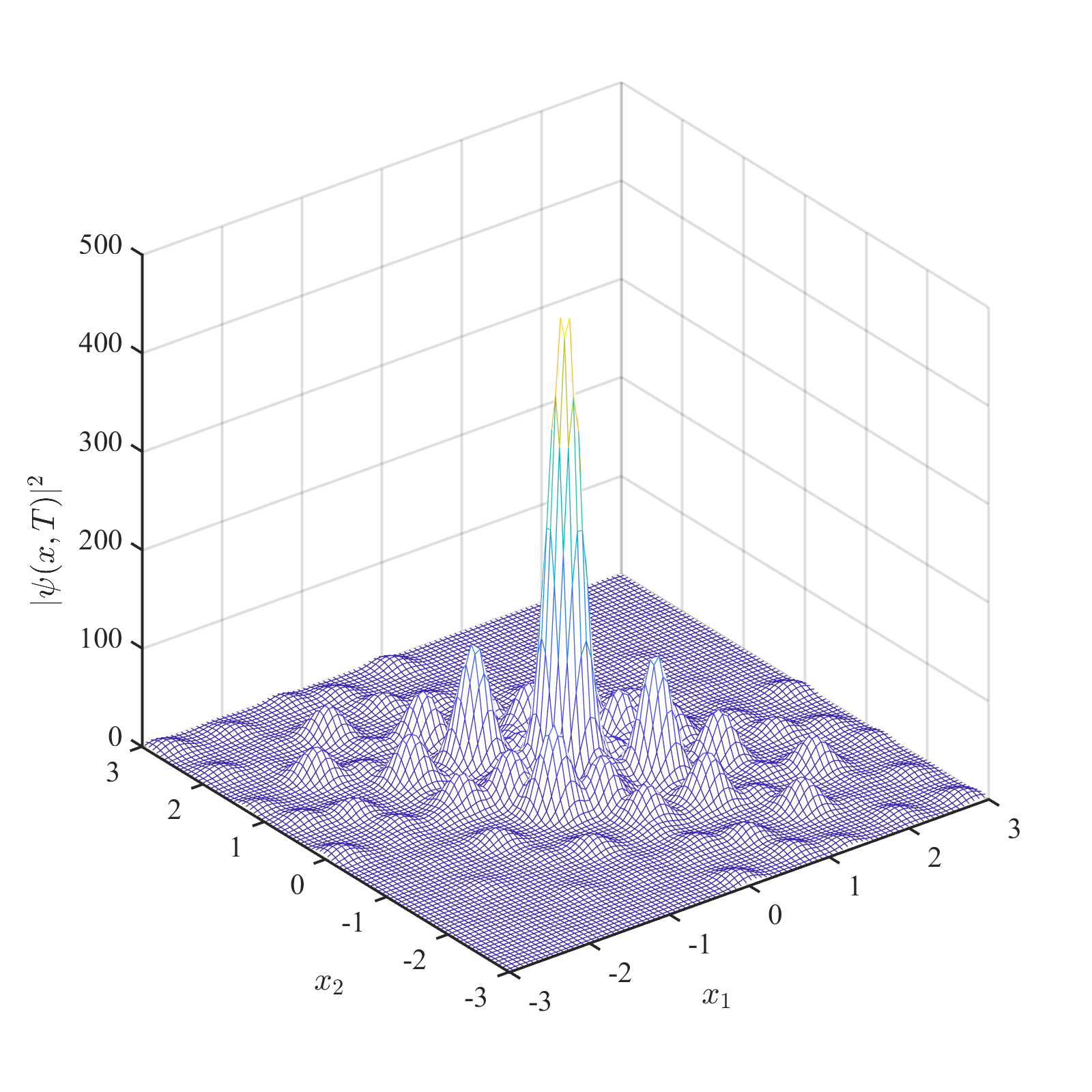}}
		\caption{Fundamental soliton located at the center.} \label{fig:soliton-EG2}
	\end{figure}

	\section{Conclusions}
	\label{sec:diss}
	
	The NQSE plays an important role in various fields, but effective numerical methods are still not available in the literature. In this paper, an efficient and accurate numerical algorithm for solving the NQSE  is presented, and rigorous error analysis is given. Compared with numerical analysis of periodic problems, many conventional analysis tools  do not apply for the quasiperiodic case, while the transfer between spaces of low-dimensional quasiperiodic and high-dimensional periodic functions and its coupling with the nonlinearity of the operator splitting scheme make the analysis intricate. 
	Many new ideas and methods are developed to address these issues, which could also promote the development of numerical analysis of other kinds of quasiperiodic systems.

	%
	%
	


	

	

\end{document}

%% file: ex_shared_NM.tex
\newcommand{\ds}{\displaystyle}

\newcommand{\by}{\bm{y}}

\newcommand{\bx}{\bm{x}}
\newcommand{\bp}{\bm{p}}

\newcommand{\bk}{\bm{k}}
\newcommand{\bj}{\bm{j}}

\newcommand{\bK}{\bm{K}}

\newcommand{\bs}{\bm{s}}

\newcommand{\bP}{\bm{P}}

\newcommand{\bmm}{\bm{m}}

\newcommand{\blam}{\bm{\lambda}}

\newcommand{\bbR}{\mathbb{R}}

\newcommand{\bbZ}{\mathbb{Z}}
\newcommand{\bbQ}{\mathbb{Q}}

\newcommand{\bbN}{\mathbb{N}}

\newcommand{\bbT}{\mathbb{T}}

\newcommand{\calC}{\mathcal{C}}

\newcommand{\calS}{\mathcal{S}}

\newcommand{\calM}{\mathcal{M}}

\newcommand{\calF}{\mathcal{F}}

\newcommand{\calI}{\mathcal{I}}

\newcommand{\hpsi}{\hat{\psi}}

\newcommand{\hU}{\hat{U}}

\newcommand{\tpsi}{\tilde{\psi}}
\newcommand{\tPsi}{\tilde{\Psi}}
\newcommand{\tphi}{\tilde{\phi}}

\newcommand{\QP}{\mbox{QP}}
\newcommand{\Err}{\mbox{Err}}

\newtheorem{thm}{Theorem}[section]
\newtheorem{lemma}[thm]{Lemma}

\newtheorem{defy}[thm]{Definition}

\newcommand\tbbint{{-\mkern -16mu\int}}

\newcommand\dbbint{{-\mkern -19mu\int}}

\newcommand\bbint{
	{\mathchoice{\dbbint}{\tbbint}{\tbbint}{\tbbint}}
}

\allowdisplaybreaks[4]

%% file: Erroranalysis-nonlinear-QSE-StrangsplittingPM.bbl
\begin{thebibliography}{99}
		
		
		
		
		
		
		
		\bibitem{Schrodinger1926undulatory}
		E. Schr\"odinger, An undulatory theory of the mechanics of atoms and molecules, Phys. Rev., 28: 1049-1070, 1926.
		
		\bibitem{Thalhammer2012Convergence}
		M. Thalhammer, Convergence analysis of high-order time-splitting pseudo-spectral methods for nonlinear Schr\"odinger equations, SIAM J. Numer. Anal., 50: 3231-3258, 2012.
		
		\bibitem{Auzinger2015Defect}
		W. Auzinger, H. Hofst\"atter, O. Koch and M. Thalhammer, Defect-based local error estimators for splitting methods, with application to Schr\"odinger equations, Part III: The nonlinear case, J. Comput. Appl. Math., 273: 182-204, 2015.
		
		\bibitem{Dirac1958principles}
		P. Dirac, The principles of quantum mechanics, Oxford University Press, 1958.
		
		\bibitem{Gauckler2010Nonlinear}
		L. Gauckler and C. Lubich, Nonlinear Schr\"odinger equations and their spectral semi-discretizations over long times, Found. Comput. Math., 10: 141-169, 2010.
		
		\bibitem{Gauckler2010Convergence}
		L. Gauckler, Convergence of a split-step Hermite method for the Gross-Pitaevskii equation, IMA J. Numer. Anal., 31: 396-415, 2011.
		
		
		\bibitem{Jahnke2000error}
		T. Jahnke and C. Lubich, Error bounds for exponential operator splittings, BIT Numer. Math., 40(4): 735-744, 2000.
		
		\bibitem{LiWu} B. Li and Y. Wu, A fully discrete low-regularity integrator for the 1D periodic cubic nonlinear Schr\"odinger equation, Numer. Math., 149(1): 151-183, 2021.
		
		\bibitem{LiZha} 
		B. Li, J. Zhang and C. Zheng, Stability and error analysis for a second-order fast approximation of the one-dimensional Schr\"odinger equation under absorbing boundary conditions,
		SIAM J. Sci. Comput., 40(6): A4083-A4104, 2018.
		
		
		\bibitem{Ablowitz2006Solitons}
		M. J. Ablowitz, B. Ilan, E. Schonbrun and R. Piestun, Solitons in two-dimensional lattices possessing defects, dislocations, and quasicrystal structures, Phys. Rev. E, 74: 035601, 2006.
		
		\bibitem{Bagci2021Soliton}	
		M. Ba\v gci, Soliton dynamics in quadratic nonlinear media with two-dimensional Pythagorean aperiodic lattices. Josa B, 38(4): 1276-1282, 2021.
		
		\bibitem{Wang2020Localization}
		P. Wang, Y. Zheng, X. Chen, C. Huang, Y. V. Kartashov, L. Torner, V. V. Konotop and F. Ye, Localization and delocalization of light in photonic Moir\'e lattices, Nature, 577: 42-46, 2020.
		
		\bibitem{Aubry1980Effects}
		S. Aubry and G. Andr\'{e}, Analyticity breaking and Anderson localization in incommensurate lattices, Ann. Israel Phys. Soc. 3: 133-164, 1980.
		
		\bibitem{Fu2020Optical}
		Q. Fu, P. Wang, C. Huang, V. Yaroslav, L. Torner, V. V. Konotop and F. Ye, Optical soliton formation controlled by angle twisting in photonic moir\'e lattices, Nat. Photonics, 14: 663-668, 2020.
		
		\bibitem{li2021numerical}
		X. Li and K. Jiang, Numerical simulation for quasiperiodic quantum dynamical systems (in Chinese), J. Numer. Meth. Comput. Appl., 42(1): 3-17, 2021.
		
		\bibitem{Shi2024Anderson}
		Y. Shi and W. Wang, Anderson localized states for the quasi-periodic nonlinear Schr\"{o}dinger equation on $\bbZ^d$, arXiv:2405.17513.	
		
		\bibitem{Avila2009spectrum}
		A. Avila, On the spectrum and Lyapunov exponent of limit periodic Schr\"{o}dinger operators,
		Commun. Math. Phys., 288: 907-918, 2009.
		
		\bibitem{Avila2015Sharp}
		A. Avila, J. You, and Q. Zhou, Sharp phase transitions for the almost Mathieu operator. Duke Math. J., 166: 2697-2718, 2015.
		
		
		\bibitem{Berti2012Sobolev}
		M. Berti and P. Bolle, Sobolev quasi-periodic solutions of multidimensional wave equations with a multiplicative potential, Nonlinearity, 25: 2579-2613, 2012.
		
		\bibitem{Bourgain1994Construction}
		J. Bourgain, Construction of quasi-periodic solutions for Hamiltonian perturbations of linear equations and applications to nonlinear PDE, Int. Math. Res. Notices, 1994(11): 475-497, 1994.
		
		\bibitem{Marx2017Dynamics}
		C. Marx and S. Jitomirskaya, Dynamics and spectral theory of quasi-periodic Schr\"{o}dinger-type operators, Ergod. Theor. Dyn. Syst., 37(8): 2353-2393, 2017.
		
		\bibitem{Shi2021Absence}
		Y. Shi, Absence  of  eigenvalues  of  analytic  quasi-periodic Schr\"{o}dinger operators  on $\bbR^d$, Commun. Math. Phys., 386: 1413-1436, 2021.
		
		\bibitem{Wang2020Space}
		W. Wang, Space quasi-periodic standing waves for nonlinear Schr\"{o}dinger equations, Commun. Math. Phys., 378(2): 783-806, 2020.
		
		\bibitem{Wang2022Infinite}
		W. Wang, Infinite energy quasi-periodic solutions to nonlinear Schr\"{o}dinger equations on $\bbR$, Int. Math. Res. Notices, rnab327, 2022.
		
		
		\bibitem{JiaZha} 
		K. Jiang and P. Zhang, Numerical methods for quasicrystals, J. Comput. Phys., 256: 428-440, 2014.
		
		\bibitem{Jiang2022approximation}
		K. Jiang, S. Li and P. Zhang, On the approximation of quasiperiodic functions with Diophantine frequencies by periodic functions, accepted by SIAM J. Math. Anal., also see arXiv:~2304.04334.
		
		\bibitem{jiang2023High-accuracy}
		K. Jiang, S. Li and J. Zhang, High-accuracy numerical methods and convergence analysis for Schr\"{o}dinger equation with incommensurate potentials. J. Sci. Comput., 101: 18, 2024.
		
		
		
		
		
		\bibitem{Bao2003Numerical}
		W. Bao, S. Jin and P. Markowich, Numerical study of time-splitting spectral discretizations of nonlinear Schr\"odinger equations in the semiclassical regimes, SIAM J. Sci. Comput., 25: 27-64, 2003.
		
		
		
		\bibitem{Bao2021Uniform}
		W. Bao, Y. Feng and C. Su, Uniform error bounds of time-splitting spectral methods for the long-time dynamics of the nonlinear Klein-Gordon equation with weak nonlinearity, Math. Comp., 91: 811-842, 2021.
		
		\bibitem{Wu2024Error}
		Z. Wu, Z. Zhang and X. Zhao, Error estimate of a quasi-Monte carlo time-splitting pseudospectral method for nonlinear Schr\"odinger equation with random potentials, SIAM/ASA J. Uncertain. Quantification 12: 1-29, 2024.
		
		
		
		\bibitem{Einkemmer2013almost}
		L. Einkemmer, A. Ostermann, An almost symmetric Strang splitting scheme for nonlinear evolution equations, Comput. Math. Appl., 67: 2144-2157, 2013.
		
		\bibitem{Einkemmer2014Convergence}
		L. Einkemmer and A. Ostermann, Convergence analysis of Strang splitting for Vlasov--type equations, SIAM J. Numer. Anal., 52: 140-155, 2014.
		
		
		
		\bibitem{Corduneanu1989Almost}
		C. Corduneanu, Almost periodic function, Chelsea, New York, 1989.
		
		
		
		
		
		
		
		\bibitem{Got} 
		S. Gottlieb and C. Wang,  Stability and convergence analysis of fully discrete Fourier collocation spectral method for 3-D viscous Burgers’ equation, J. Sci. Comput., 53: 102-128, 2012.
		
		
		
		
		\bibitem{Jiang2024Numerical}
		K. Jiang, S. Li and P. Zhang, Numerical methods and analysis of computing quasiperiodic systems, SIAM J. Numer. Anal., 62(1): 353-375, 2024.
		
		
		
		\bibitem{Shen2011Spectral}
		J. Shen, T. Tang and L. Wang, Spectral methods: algorithms, Analysis and Applications, Springer, Berlin, 2011.
		
		
		
		\bibitem{Stone1930linear}
		M. Stone, Linear transformations in Hilbert Space: III. operational methods and group theory. Proc. Natl. Acad. Sci. USA, 16(2): 172-175, 1930.
		
		\bibitem{frigo2005design}
		M. Frigo and S. Johnson, The design and implementation of FFTW3. Proc. IEEE, 93: 216-231, 2005.
		
		
		
		
		
		
		
		
		
		
		
		
		
		
		
		
		
		
		
		
		
		
		
		
		
		
		
		
		
		
	\end{thebibliography}
